\documentclass[12pt,leqno]{article}
\usepackage{amssymb,amsthm,amsmath,graphicx}
\usepackage[lite]{amsrefs}

\newtheorem{theorem}{Theorem}[section]

\newtheorem{lemma}[theorem]{Lemma}

\newtheorem{remark}[theorem]{Remark}
\newtheorem{definition}[theorem]{Definition}

\newtheorem{corollary}[theorem]{Corollary}

\newcommand{\thm}[1]{Theorem~\ref{#1}}
\newcommand{\lem}[1]{Lemma~\ref{#1}}
\newcommand{\cor}[1]{Corollary~\ref{#1}}

\newcommand{\sectio}[1]{Section~\ref{#1}}

\newcommand{\wx}{\widetilde x}
\newcommand{\W}{{\mathsf w}}
\newcommand{\tPhi}{\widetilde\Phi}

\newcommand{\uar}{%
\mathrel{\raisebox{.1em}{%
\reflectbox{\rotatebox[origin=c]{90}{$\rightsquigarrow$}}}}}

\newcommand{\dar}{%
\mathrel{\raisebox{.1em}{%
\reflectbox{\rotatebox[origin=c]{270}{$\rightsquigarrow$}}}}}

\DeclareMathOperator{\Lip}{Lip}
\DeclareMathOperator{\dist}{dist}

\newcommand{\sect}[1]{\section{#1} \setcounter{equation}{0} }

\newcommand{\norm}[2]{\left\|#1\right\|_{#2}}

\newcommand{\Poly}{\Pi}
\newcommand{\Pn}{\Poly_n}

\newcommand{\N}{\mathbb N}
\newcommand{\R}{\mathbb R}
\newcommand{\DD}{{\mathcal D}}

\newcommand{\DDD}{{\mathfrak d}}
\newcommand{\K}{{\mathcal K}}
\newcommand{\KK}{\K_n(f, A, k,r)}
\newcommand{\tj}{{\mathfrak z}}

\newcommand{\andd}{\quad\mbox{\rm and}\quad}
\newcommand{\sn}{|f|_{\Lip^*\alpha}}
\newcommand{\co}{\vartheta}

\newcommand\w{{\omega}}
\newcommand\vare{\varepsilon}
\newcommand\eps{\epsilon}
\newcommand\we{{\widetilde\varepsilon\,}}

\def\be  {\begin{equation}}
\def\ee  {\end{equation}}

 \newcommand\uw{\mathrm{w}}

\newcommand{\ineq}[1]{(\ref{#1})}

\newcommand{\ie}{i.e., }
\newcommand{\eg}{e.g., }
\newcommand{\suchthat}{\;\; \big| \;\;}

\title{{\sc Exact order of pointwise estimates for polynomial approximation with Hermite interpolation
}\thanks{{\it AMS classification:} Primary 41A25, 41A10;  Secondary 41A05, 41A17, 41A28 {\it Keywords
and phrases:} Hermite interpolation,  simultaneous approximation,  moduli of smoothness, exact estimates, exact orders,  interpolatory estimates, approximation with Hermite interpolation of Sobolev and Lipschitz classes }}

\author{K. A. Kopotun\thanks
{Department of Mathematics, University of Manitoba,Winnipeg, Manitoba, R3T 2N2, Canada ({\tt
kirill.kopotun@umanitoba.ca}). Supported in part by NSERC grant RGPIN-05678-2020.}\ \
D. Leviatan\thanks{Raymond and Beverly Sackler School of Mathematical
Sciences, Tel Aviv University, Tel Aviv 6139001, Israel ({\tt
leviatan@tauex.tau.ac.il}).}\ \
and \ I. A. Shevchuk\thanks
{Faculty of Mechanics and Mathematics, Taras
Shevchenko National University of Kyiv, 01601 Kyiv, Ukraine ({\tt
shevchuk@univ.kiev.ua}). Supported by the National Research Foundation of Ukraine, Project  \#2020.02/0155.}
}

\begin{document}

\maketitle

\abstract{ 
We establish best possible pointwise (up to a constant multiple) estimates for  approximation, on a finite interval, by polynomials that satisfy finitely many (Hermite) interpolation conditions, and show that these estimates cannot be improved.
In particular, we show that {\bf any} algebraic polynomial of degree $n$ approximating  a function $f\in C^r(I)$, $I=[-1,1]$, at the classical pointwise rate $c(k,r)\rho_n^r(x) \w_k(f^{(r)}, \rho_n(x))$, where $\rho_n(x)=n^{-1}\sqrt{1-x^2}+n^{-2}$, and $c(k,r)$ is a constant which depends only on $k$ and $r$, and is independent of $f$ and $n$; and  (Hermite) interpolating $f$ and its derivatives up to the order $r$ at a point $x_0\in I$, has the best possible pointwise rate of (simultaneous) approximation of $f$   near $x_0$.
 Several applications are given.
}


\sect{Introduction and main results}

The main theme of this paper is polynomial approximation of functions on a finite interval, where we impose on the polynomials finitely many interpolation conditions  (including Hermite interpolation).
%
 Stated simply, the question we discuss is how well   a continuous (or continuously differentiable) function can be approximated by algebraic polynomials if it is required that the polynomials also interpolate this function (and perhaps its derivatives) at a given set of points. We are especially interested in  the improvement of the rate of approximation  near these interpolation points.
 Clearly, due to the interpolation, this rate   may be improved, but is there a limit to this improvement? The main purpose of this paper is to provide exact answers, that is, to establish these types of estimates and to show that they cannot be improved.

\subsection{Motivation}

 We start by recalling some standard notation. As usual, $C^r(J)$ denotes the space of $r$ times continuously differentiable functions on $J$,  $C^0(J):=C(J)$ is the space of continuous functions on $J$, equipped with the uniform norm which   will be denoted by $\norm{\cdot}{J}$. For $k\in\N$ and an interval $J$,
$\Delta^k_u(f,x;J):= \sum_{i=0}^k(-1)^i  \binom ki    f(x+(k/2-i)u)$ if $x\pm ku/2\in J$ and $:=0$, otherwise, and
$
\w_k(f,t;J):=\sup_{0<u\le t}\|\Delta^k_u(f,\cdot;J)\|_{J}
$
is the $k$th modulus of smoothness of $f$ on $J$.
When dealing with the interval $I:=[-1,1]$, we suppress referring to the interval and  use the notation  $\|\cdot\|:=\|\cdot\|_{I}$,  $\omega_k(f,t):=\omega_k(f,t;I)$, $C^r := C^r(I)$,  etc.

Also,
\[ 
\varphi(x):=\sqrt{1-x^2}, \quad \rho_n(x):= \varphi(x)n^{-1}+ n^{-2}, \; n\in\N, 
\]
$\rho_0(x)\equiv 1$, and  $\Pn$ denotes  the space of algebraic polynomials of degree $\le n$. Note that $\varphi(x) n^{-1} \sim \rho_n(x)$ for $x\in I\setminus S_n$, where $S_n:=   [-1,-1+n^{-2}] \cup  [1-n^{-2}, 1]$.

 Recall the classical Timan-Dzyadyk-Freud-Brudnyi direct theorem for the approximation by algebraic polynomials (see  \cites{D-1958, F-1959, B,T-1951}):
if $k\in\N$, $r\in\N_0$ and $f\in C^r$, then for each $n\geq k+r-1$ there is a polynomial $P_n\in\Pn$ satisfying
\be \label{classdir}
|f(x)-P_n(x)| \le c(k,r) \rho_n^r(x) \w_k(f^{(r)}, \rho_n(x)) , \quad x\in I .
\ee

As one can  see in  \ineq{classdir}, the order of approximation becomes significantly better near the endpoints of $I$ than in the middle. One might think that if   the approximating polynomials happen to {\em interpolate} $f$ at the endpoints, then, perhaps, the estimates become even better (sometimes, these types of estimates are called {\em interpolatory (pointwise) estimates}). This turns out to be correct, however, there is a limit to how much they can improve.

The following is a brief history of development of interpolatory estimates for polynomial approximation on $I$. There is also a parallel history of development of interpolatory estimates for {\em simultaneous} approximation,
but we do not discuss it here in order to avoid the possibility of confusing readers with technical details and restrictions on parameters involved which, as will be seen later,
are an artifact of a specific form of the estimates and may be avoided. (Interested readers can refer to \cite{K-sim}*{pp. 68, 69} for further discussions.)

In 1963,  Lorentz \cite{Lor}  asked if quantity $\rho_n(x)$ in \ineq{classdir} can be replaced by $\varphi(x)n^{-1}$.  He actually posed the problem in the case $k=1$ only, but it is understandable since Brudnyi \cite{B}   published his proof of \ineq{classdir} for $k>1$ in 1963, and so Lorentz may not have been aware of this new development when he posed his problem in \cite{Lor} (although \ineq{classdir} in the case $k=2$ was known as early as 1958, see \cites{D-1958,F-1959}).
Hence, the following open problem was offered to the research community in 1963:
determine all pairs $(k,r)\in\N\times\N_0$, such that for any $f\in C^r$ and $n\geq c(k,r)$ there is a polynomial $P_n\in\Pn$ satisfying
\be \label{interpol}
|f(x)-P_n(x)| \le c(k,r)  \varphi^r(x)n^{-r}   \w_k(f^{(r)}, \varphi(x)n^{-1}) , \quad x\in I .
\ee

In 1966, Teljakovski\u{i} \cite{Tel} proved \ineq{interpol} in the case $k=1$
and, in 1967,  Gopengauz \cite{Gop} published a stronger result on simultaneous approximation  yielding Teljakovski\u{i}'s theorem as a corollary. One remark is in order here. Teljakovski\u{i}'s paper \cite{Tel} was submitted on 16 April 1965, and there is the following footnote in Gopengauz's paper \cite{Gop}: ``Original article submitted September 30, 1966. This article was received from
the editorial staff of the journal {\em Uspekhi matematicheskikh nauk} (where it was submitted October 30, 1964)
in connection with the liquidation of this journal's division {\em Scientific reports and problems}.'' Hence, it seems that Gopengauz's paper was prepared for publication before Teljakovski\u{i}'s.
%
Estimates similar to \ineq{interpol} are usually referred to as {\em Teljakovski\u{i}-Gopengauz type} estimates.
In 1975, DeVore \cites{D,De76} proved \ineq{interpol} in the case $(k,r)=(2,0)$ being the first to prove a Teljakovski\u{i}-Gopengauz estimate for $k>1$.
In 1983, Hinnemann and Gonska \cite{HG} extended DeVore's result by establishing \ineq{interpol} in the case $k=2$ and $r\in \N$.
In 1985, they \cite{GH} also showed the validity of \ineq{interpol} for $k\le r+2$.
In 1985,   Yu \cite{Yu} showed that \ineq{interpol}  is not true if  $k \ge r + 3$,  thus showing that there is a limit to improvement (this was later strengthened in \cite{GLSW}*{Theorem 1}).

To summarize the above, \ineq{interpol} is valid if and only if
\[ 
(k,r) \in \Upsilon := \left\{ (k,r) \suchthat k\in N, r\in N_0 \andd k\le r+2 \right\}.
\]

Ever since, there has been great interest in Teljakovski\u{i}-Gopengauz--type estimates for problems dealing with:
simultaneous approximation of a function and its derivatives \cites{Da, AB, Li, K-sim},
(generalized) discrete linear polynomial operators satisfying this type of estimates \cites{Sz,BG,KS86},
polynomial approximation with extra interpolation conditions \cites{V,XZ},
or  with extra Hermite interpolation conditions \cites{KP96,Trigub, AB, TIzv}. 

We now note that, while \ineq{interpol} is perhaps aesthetically pleasing, there is no reason why the quantity $\rho_n(x)$ in \ineq{classdir} should be replaced by $\varphi(x)n^{-1}$ and not something which approaches zero at $\pm 1$ faster than with order $1/2$ but is ``the same as $\rho_n(x)$ in the middle of $I$'' (as it is well known that $\rho_n(x)$ in \ineq{classdir} cannot  essentially be replaced by a smaller quantity away from the endpoints of $I$).

First, we have the following theorem which is an immediate consequence of \cite{K-sim}*{Corollary 2-3.4} (proved for simultaneous approximation there).

\begin{theorem}[\mbox{see \cite{K-sim}*{Corollary 2-3.4}}]\label{thk-sim}%
Let $r\in\N_0$, $k\in\N$ and $f\in  C^r$. Then for any $n\ge \max\{k+r-1, 2r+1\}$, there is a polynomial $P_n \in \Pn$ such that
\ineq{classdir} is valid and, moreover,
for $x\in S_n:=[-1,-1+n^{-2}] \cup  [1-n^{-2}, 1]$, the following improved estimate  holds
\be \label{sim2}
|f(x)-P_n(x)| \le c(r,k) \varphi^{2r}(x)  \w_k(f^{(r)}, \varphi^{2/k}(x) n^{-2+2/k} ).
\ee
\end{theorem}

It follows from \cite{K-sim}*{Theorem 3}    that, for any $\gamma \in\R$ and $n\in\N$, the quantity  $\varphi^{2/k}(x)  n^{-2+2/k}$ in \ineq{sim2} cannot be replaced by $\varphi^{2\beta}(x) n^\gamma$ with $\beta > 1/k$. In fact, for any $\gamma_1, \gamma_2 \in\R$, if $\alpha + k\beta > r+1$, then one cannot replace the right-hand side of \ineq{sim2} by $c(k,r) \varphi^{2\alpha}(x) n^{\gamma_1} \w_k(f^{(r)}, \varphi^{2\beta}(x) n^{\gamma_2})$. Hence, \ineq{sim2} is exact in this sense (see also \cor{negativecor} below, with $x_0=\pm 1$, for a stronger negative result).

Since $\varphi(x) n \le \sqrt{2}$, $x\in S_n$,  and $\varphi(x) = o(1)$, $x\to \pm 1$, then
using the well know property of moduli of smoothness $\w_k(g, \lambda t) \le \lceil \lambda \rceil^k \w_k(g, t)$, $\lambda >0$,
one can immediately see that \ineq{sim2} is essentially better than (except for a few particular cases when it is as good as)   \ineq{interpol} if $(k,r)\in\Upsilon$. At the same time, if $(k,r)\not\in\Upsilon$, then \ineq{sim2} is still valid while \ineq{interpol} is not.

Because of this, it seems to be clear that it is advantageous to always discuss interpolatory estimates in the form similar to \ineq{sim2} instead of Teljakovski\u{i}-Gopengauz type estimates of type \ineq{interpol}.\footnote{In 2002, Trigub \cite{Trigub}*{Theorem 1} published exactly the same theorem as \cite{K-sim}*{Corollary 2-3.4} that appeared in 1996.}

It follows from  \ineq{sim2} that $|f(x)-P_n(x)| \varphi^{-2r}(x) \to 0$, $x\to \pm 1$, and hence $P_n$ from the statement of \thm{thk-sim} has to be such that
  $P_n^{(j)}(\pm 1)  =  f^{(j)}(\pm 1)$ for all $0\le j \le r$.
It turns out (see \thm{main} below) that  any polynomial $P_n$ satisfying these interpolation conditions as well as \ineq{classdir} also satisfies \ineq{sim2} and, moreover, it satisfies \ineq{sim2} for all moduli of smoothness of lower order as well. (After we  prepared this paper for publication, we had discovered that Bal\'{a}zs and Kilgore \cite{BK95} used a somewhat similar idea to provide an alternative proof of Gopengauz's result \cite{Gop}.)
  We emphasize that while it is a classical property of the moduli of smoothness that $\w_k(g,t) \le 2^{k-\ell}\w_\ell(g,t)$ for $1\le \ell\le k$, it is not true that $\phi_k(x,n) :=  \w_k(f^{(r)}, \varphi^{2/k}(x) n^{-2+2/k})$ can be estimated above by $\phi_\ell(x,n)$  because there exist functions $f\in C^r$ for which $\lim_{x\to \pm 1} \phi_k(x,n)/ \phi_\ell(x,n) =\infty$ (see \cite{KLSconspline}*{Remark 1.3}).

In several theorems below, we assume that a polynomial whose improved rate of approximation we are discussing satisfies the classical estimate \ineq{classdir} with a certain constant $A\ge 1$ instead of $c$. Instead of restating this inequality every time it is needed, we will say that a polynomial is from the class $\KK$: 

\begin{definition} 
Given $k\in\N$, $r,n\in\N_0$, $f\in C^r$ and $A\ge 1$, we say that a polynomial $P_n$ belongs to the class $\KK$ if $P_n\in\Pn$ and
\be \label{classKK}
|f(x)-P_n(x)|\le A\rho_n^r(x)\w_k(f^{(r)},\rho_n(x)),\quad x\in I .
\ee
\end{definition}

The following result on interpolatory pointwise estimates for {\em simultaneous} approximation is a consequence of a more general \thm{mainnew} below and is the main application of our main results to the type of problems that we discussed above.

\begin{theorem} \label{main}
 Let $k \in\N$, $r,n\in\N_0$ and  $f\in C^r$, and suppose that $P_n\in \KK$ is such that
 \[ 
P_n^{(j)}(-1)=f^{(j)}(-1),\quad P_n^{(j)}(1)=f^{(j)}(1),\quad \text{for }\; 0\le j \le r .
\]
Then,  for all $0\le \nu \le r$ and  $1\le \ell \le k$,
\be   \label{1.8}
|f^{(\nu)}(x)-P_n^{(\nu)}(x)|\le c(k,r) A\varphi^{2(r-\nu)}(x)\w_\ell(f^{(r)},\varphi^{2/\ell}(x)n^{-2+ 2/\ell}),
\ee
if $1-n^{-2}\le| x|\le1$.
\end{theorem}
While \thm{main} seems stronger than an analogous result in the non simultaneous case (\ie the case $\nu=0$ in \ineq{1.8}),
 it is actually an added benefit of our approach that results of interpolatory type for simultaneous approximation of a function and its derivatives immediately follow from non-simultaneous ones because of the following
 rather well known lemma (we provide its short proof in Section~\ref{simultsect} for completeness).

\begin{lemma} \label{traux}
Let $k\in\N$,   $r,n \in\N_0$ and  $f\in C^r$. If $P_n\in \KK$,
 then, for all $x\in I$, we have
\be \label{tr1}
|f^{(\nu)}(x)- P_n^{(\nu)}(x)| \le c(k,r) A \rho_n^{r-\nu}(x) \w_k(f^{(r)}, \rho_n(x)), \quad 0\le \nu\le r ,
\ee
and
\be \label{tr2}
| P_n^{(k+r)}(x)| \le c(k,r) A \rho_n^{-k}(x)  \w_{k } (f^{(r)}, \rho_n(x)) .
\ee
\end{lemma}

We are now ready to start discussing our main results which are much more general and, because of that, a bit more technical.

\subsection{Main results}

We begin with the following theorem which shows that {\em any} polynomial from $\Pn$ approximating a function $f\in C^r$ so that the classical direct estimate  holds and interpolating $f$ and its derivatives at some point,
 has the best possible pointwise rate of  approximation of $f$   near that point.

\begin{theorem} \label{maingen}
 Let $k\in\N$, $x_0\in I$,  $r,n\in\N_0$,   $f\in C^r$, and $m\in\N_0$ is such that $m\le r$. If $P_n\in \KK$ satisfies
 \be \label{2g}
P_n^{(j)}(x_0)=f^{(j)}(x_0), \quad \text{for }\; 0\le j \le m,
\ee
then, for $x\in I$,
\begin{align}  \label{4g}
|&f(x)-P_n(x)| \nonumber \\
& \le c(k,r) A
\begin{cases}
|x-x_0|^{m+1} \rho_n^{r-m-1}(x)  \omega_k(f^{(r)}, \rho_n(x)) , & \text{if }\; m \le r-1 ,\\
|x-x_0|^r \omega_k(f^{(r)},|x-x_0|^{1/k} \rho_n^{1-1/k}(x)) , & \text{if }\; m=r .
\end{cases}
\end{align}
\end{theorem}

We remark that 
\ineq{4g} is stronger than \ineq{classKK}  if  $|x-x_0| = o(\rho_n(x))$, $x\to x_0$, and is as good as or weaker otherwise. Hence, \ineq{4g} only becomes useful when $x$ is sufficiently close to $x_0$.   Also, if $|x-x_0| \le \rho_n(x)$, then $\rho_n(x) \sim \rho_n(x_0)$, and so $\rho_n(x)$ in \ineq{4g} can be replaced by $\rho_n(x_0)$ for these $x$.

Using  \lem{traux} and the inequality $\omega_k(f^{(r)},t)\le 2^{k-\ell}\omega_\ell(f^{(r)},t)$, $1\le \ell \le k$,  we conclude that \thm{maingen} immediately implies  the following  result on simultaneous approximation.

\begin{theorem} \label{mainnew}
 Let $k \in\N$, $x_0\in I$,  $r,n\in\N_0$,   $f\in C^r$, and $m\in\N_0$ is such that $m\le r$. If $P_n\in \KK$ satisfies \ineq{2g} then,
  for all $0\le \nu \le r$, $1\le \ell \le k$ and $x\in I$, we have
\begin{align}  \label{4nnn}
| & f^{(\nu)}(x)  -P_n^{(\nu)}(x)| \nonumber \\
 & \le c(k,r) A
\begin{cases}
|x-x_0|^{\sigma} \rho_n^{r-\nu-\sigma}(x)  \omega_\ell(f^{(r)}, \rho_n(x)) , & \text{if }\; m \le r-1 ,\\
|x-x_0|^{r-\nu} \omega_\ell(f^{(r)},|x-x_0|^{1/\ell} \rho_n^{1-1/\ell}(x)) , & \text{if }\; m=r ,
\end{cases}
\end{align}
where $\sigma := \max\{ m-\nu+1, 0\}$.
\end{theorem}

Clearly, estimates \ineq{4g} and \ineq{4nnn} in Theorems~\ref{maingen} and \ref{mainnew} cannot be improved if $m\le r-1$. Indeed, if $|f^{(\nu)}(x)-P_n^{(\nu)}(x)| = o(|x-x_0|^{\sigma})$, $x\to x_0$, then
$f^{(\nu+\sigma)}(x_0) = P_n^{(\nu+\sigma)}(x_0)$ with $\nu+\sigma =  \max\{ m +1, \nu\} \ge m+1$, which does not have to be the case by \ineq{2g}.
The fact that no improvements can be made in the case $m=r$ either follows
 from the following theorem
 (see also simpler but weaker \thm{weaknegative} below
as well as discussions in Section~\ref{secweak}).

\begin{theorem}[negative theorem]\label{negative}
Let $k\in\N$,  $r\in\N_0$, $x_0\in I$, and let a positive function $\vare \in C(0,1]$ be such that $\lim_{x\to 0^+} \vare(x)=0$. Then,
there is a function $F\in C^{r} $, such that for any algebraic polynomial $P$ we have
\[ 
\limsup_{x\to x_0}\frac{|F(x)-P(x)|}{\omega_k\left(F^{(r)},\vare(|x-x_0|) |x-x_0|^{(r+1)/k}\right)}=\infty.
\]
\end{theorem}

Using $\w_k(F^{(r)}, \lambda t) \le 2^k \lambda^k \w_k(F^{(r)}, t)$, $\lambda \ge 1$, which implies, for $\beta \le (r+1)/k$,
\begin{align*}
\w_k\big(   F^{(r)}  &  ,   \vare(|x-x_0|)  |x-x_0|^{\beta}\big) \\
& \le  c |x-x_0|^{k\beta - r- 1}
\w_k\big(F^{(r)},\vare(|x-x_0|) |x-x_0|^{(r+1)/k}\big) ,
\end{align*}
 we immediately get the following corollary.

\begin{corollary} \label{negativecor}
Let $k\in\N$,  $r\in\N_0$, $x_0\in I$, and let a positive function $\vare \in C(0,1]$ be such that $\lim_{x\to 0^+} \vare(x)=0$. Then,
there is a function $F\in C^{r} $, such that for any algebraic polynomial $P$ and any $\alpha\ge 0$ and $\beta\in \R$ such that $\alpha + k\beta = r+1$ we have
\[ 
\limsup_{x\to x_0}\frac{|F(x)-P(x)|}{|x-x_0|^\alpha  \omega_k\left(F^{(r)},\vare(|x-x_0|) |x-x_0|^{\beta}\right)}=\infty.
\]
\end{corollary}

In order to discuss the general results that yield somewhat more general  estimates than those in Theorems~\ref{maingen} and \ref{mainnew}, we need to recall some definitions.
Given a collection of $s$ points $Y=\{y_j\}_{j=0}^{s-1}$  with possible repetitions, $y_0\le y_1 \le \dots \le y_{s-1}$, we recall that, for each $j$, the multiplicity $m_j$ of $y_j$ is the number of $y_i$ such that $y_i=y_j$. Also, we let $l_j$ be the number of $y_i=y_j$ with $i\le j$. Suppose that a function $f$ is defined at all points in $Y$ and, moreover, for each $y_j \in Y$, $f^{(l_j-1)}(y_j)$ is defined as well. In other words, $f$ has $m_j-1$ derivatives at each point that has multiplicity $m_j$. Then, there is a unique Lagrange-Hermite polynomial $L_{s-1}(\cdot; f, Y)$ of degree $\le s-1$ that satisfies
\be \label{hermite}
L_{s-1}^{(l_j-1)}(y_j; f, Y) = f^{(l_j-1)}(y_j), \quad \text{for all }\; 0\le j \le s-1 .
\ee

It will also be convenient for us to think about $Y$ as a set of distinct points   $z_0<z_1 <\dots < z_{\mu-1}$   with multiplicities $m_0, \dots, m_{\mu-1}$, \ie
\[
Y = \left\{ \underbrace{z_0, \dots, z_0}_{m_0}, \underbrace{z_1, \dots, z_1}_{m_1}, \dots, \underbrace{z_{\mu-1}, \dots, z_{\mu-1}}_{m_{\mu-1}} \right\} .
\]
Then, $s= m_0 + \dots + m_{\mu-1}$,
\[
y_i = z_j , \quad \text{for }\;  \sum_{l=0}^{j-1} m_l \le i < \sum_{l=0}^{j} m_l , \quad 0\le j \le \mu-1,
\]
where $\sum_{l=0}^{-1} m_l := 0$,
 and the polynomial $L_{s-1}$ satisfies
\[ 
L_{s-1}^{(j)}(z_i; f, Y) = f^{(j)}(z_i), \quad \text{for all}\quad  0\le i \le \mu-1  \andd  0\le j\le m_i-1.
\]
From now on, $Z(Y) =\{z_i\}_{i=0}^{\mu-1}$ will always denote the subset of all {\em distinct} points in $Y$.
We also often use the notation
\be \label{lambda}
\Lambda_r(Y) :=   \min_{0\le j \le s-r-2}  (y_{j+r+1}-y_j) , \quad \text{if }\; s\ge r+2,
\ee
and
\be \label{delta}
\delta(Y):= \delta(Z(Y)) :=  \min_{0\le i \le \mu-2} (z_{i+1}-z_i) , \quad \text{if }\; \mu\ge 2 .
\ee
Some explanations are needed in order to understand what these constants represent and conditions that are put on the set $Y$ if it is assumed that they are bounded away from zero. Condition $\Lambda_r(Y) >0$ means that {\em at most $r+1$} consecutive points from $Y$ are allowed to coalesce and so, in particular, the Lagrange-Hermite polynomial $L_{s-1}(\cdot; f, Y)$ is well defined if $f$ is assumed to have $r$ derivatives on $I$.
If $\Lambda_r(Y) \ge \lambda$, then the diameter of any set of $r+2$ consecutive points from $Y$ is at least $\lambda$
or, equivalently, points $z_i, \dots, z_{i+\ell}$ can all lie inside an interval of length $\lambda$ only if $m_i+\dots + m_{i+\ell} \le r+1$. This guarantees that the rate of approximation of $f$ by $L_{s-1}(\cdot; f, Y)$ will not get out of control.
To give a simple example, suppose that  $Y$ consists of two distinct points $y_0=0$ and $y_1=\epsilon$,
$f(x) = x_+^\gamma$, $\gamma>0$, and let $L_{1}(\cdot; f, Y)$ be the linear polynomial interpolating $f$ at $y_0$ and $y_1$. Then, $\norm{f}{}=1$ and
\[
\lim_{\epsilon \to 0^+} \norm{f-L_1(\cdot; f, Y)}{} =
\begin{cases}
1 , & \text{if }\; \gamma \ge 1 ,\\
\infty , & \text{if }\; 0<\gamma < 1 .
\end{cases}
\]
In other words, while it is acceptable for $y_0$ and $y_1$ to be close to each other if $\norm{f'}{L_\infty(I)} <\infty$ (the case corresponding to $r=1$), it may cause problems if $\norm{f'}{L_\infty(I)} =\infty$ (the case corresponding to $r=0$).

We also note that, for any $Y$ with at least $r+2$ points, $\Lambda_r(Y) \ge \delta(Y)$ and, in general, $\delta(Y)$ can be much smaller than $\Lambda_r(Y)$.

If  $s\le r+1$ or $\mu\le 1$ (\ie conditions on $s$ and $\mu$ in \ineq{lambda} and \ineq{delta}  are not satisfied), then we will not need to put any restrictions on the sets $Y$ or $Z(Y)$, and so they can be arbitrary subsets of $I$ with $s$ or $\mu$ points, respectively.

We will also need to refer to various subsets of points from $Y$ which are closest to a point $x\in I$, and so we introduce the following notation. Given $Y=\{y_j\}_{j=0}^{s-1}$ (recall that points $y_j$'s are allowed to coalesce) and $x\in I$, we renumber the points $y_j$ so that the distance from these points to $x$ becomes nondecreasing, \ie we   let $\sigma = (\sigma_0, \dots, \sigma_{s-1})$ be a (in general, non-unique) permutation of $(0, \dots, s-1)$ such that
\[
|x-y_{\sigma_{j-1}}| \le |x-y_{\sigma_j}| , \quad \text{for all }\; 1\le j \le s-1 .
\]
Clearly, $\sigma$ as well as all $\sigma_j$'s depend on $x$, and we use the notation ``$\sigma(x)$'' and ``$\sigma_j(x)$'' to emphasize this fact.
We also denote
\[ 
\DD_m(x, Y) := \prod_{j=0}^m |x-y_{\sigma_j(x)}| , \quad 0\le m \le s-1 .
\]
Thus, for example, $\DD_0(x,Y) = \dist(x, Y)$, $\DD_{s-1}(x,Y) = \prod_{j=0}^{s-1}|x-y_j|$ and, if a point $y_j$ has multiplicity $m_j$, then
$\DD_\nu(x,Y)=|x-y_j|^{\nu+1}$, $0\le \nu \le m_j-1$, if $x$ is sufficiently close to $y_j$.

\begin{theorem} \label{mainnew1}
Let
 $k,s\in\N$,  $r,n\in\N_0$,   $f\in C^r$, and let $Y=\{y_j\}_{j=0}^{s-1} \subset I$ be such that, if $s\ge r+2$, then $\Lambda_r(Y)$ is strictly positive.
If  $P_n\in \KK$ satisfies
\be\label{333}
P_{n}^{(l_{j}-1)}(y_{j}) = f^{(l_j-1)}(y_{j}), \quad \text{for all }\; 0\le j \le s-1,
\ee
then, for any $x\in I$, we have
 \begin{align} \label{78}
| & f(x)-P_n(x)|   \nonumber \\
&\le c(k,r)A
\begin{cases}
\DD_{s-1}(x, Y)\rho_n^{r-s}(x)  \omega_k(f^{(r)}, \rho_n(x)), & \text{if }  s \le r  ,\\
\DD_{r-1}(x, Y) \omega_k(f^{(r)},|x-y_{\sigma_r(x)}|^{1/k} \rho_n^{1-1/k}(x)), &\text{if }  s \ge r+1.
\end{cases}
\end{align}
\end{theorem}

Note that \thm{maingen} is a simpler restatement of \thm{mainnew1} in the case $s=m+1 \le r+1$ and $y_0 = \dots = y_{s-1}=x_0$. At the same time, \thm{maingen} is almost (but not quite) as general as
\thm{mainnew1} because, if a point $z_i \in Z(Y)$ has multiplicity $m_i$ in $Y$, $n\in\N$ is so large that $2\rho_n(z_i) \le \delta(Y)$,
 and $x$ is sufficiently close to $z_i$, \ie $|x-z_i|\le \rho_n(z_i)$, then
$\DD_{j}(x,Y) \ge |x-z_i|^{m_i} \rho_n^{j+1-m_i}(z_i)$ for $j \ge m_i$,
 and $y_{\sigma_r(x)} = z_i$ if $m_i=r+1$, or $|x-y_{\sigma_r(x)}| \ge \rho_n(z_i)$ if $m_i \le r$.
Hence, for each  $0\le i\le \mu-1$,      \thm{maingen}  with $x_0=z_i$ and $m=m_i-1$ yields the estimate \ineq{78} for all $x\in [z_i-\rho_n(z_i), z_i+\rho_n(z_i)]\cap I$, and, of course,
if $\dist(x, Z) = |x-z_i| \ge  \rho_n(z_i)$, then  \ineq{78} is weaker than \ineq{classKK}.

In other words, if $n$ is sufficiently large depending on $\delta(Y)$ (for example, if $n \ge 4/\delta(Y)$), then there is no difference between Theorems~\ref{maingen} and \ref{mainnew1}. However, if $n$ is ``small' then \thm{mainnew1} is stronger.

All (positive) results above assume that we work with a polynomial from  the class  $\KK$ that also satisfies  Hermite interpolation conditions of type \ineq{333}, and we will show that such polynomials exist with some constant $A$ depending only on $k$, $r$ and $s$.
 The following theorem is proved in Section~\ref{thproof}.

\begin{theorem} \label{another}
Let $k,s \in\N$, $r\in \N_0$, and suppose that a set $Y=\{y_j\}_{j=0}^{s-1} \subset I$ is such that, if $s\ge r+2$, then
$\Lambda_r(Y) \ge \lambda >0$.
If $f\in C^r$ then, for every $n \ge N(k, r, s, \lambda)$, there exists a polynomial $P_n \in\Pn$ such that
\[ 
P_{n}^{(l_j-1)}(y_j; f, Y) = f^{(l_j-1)}(y_j), \quad \text{for all }\; 0\le j \le s-1 ,
\]
and
\be \label{an2}
|f(x)-P_n(x)|\le c(k,r,s) \rho_n^r(x)\w_k(f^{(r)},\rho_n(x)),\quad x\in I.
\ee
\end{theorem}

Combining Theorems~\ref{another} and \ref{mainnew1} we arrive at the following general result on Hermite interpolation.
\begin{theorem} \label{191}
Let $k,s \in\N$, $r\in \N_0$, and suppose that a set $Y=\{y_j\}_{j=0}^{s-1} \subset I$ is such that, if $s\ge r+2$, then
$\Lambda_r(Y) \ge \lambda >0$.
If $f\in C^r$ then, for every $n \ge N(k, r, s, \lambda)$, there exists a polynomial $P_n \in\Pn$ such that, for all $x\in I$,
\[
|f(x)-P_n(x)|\le c(k,r,s) \rho_n^r(x)\w_k(f^{(r)},\rho_n(x)),
\]
and, moreover,
\begin{align}  \label{an222}
&|f(x)-P_n(x)| \nonumber \\
& \le c(k,r,s)
\begin{cases}
\DD_{s-1}(x, Y)\rho_n^{r-s}(x)  \omega_k(f^{(r)}, \rho_n(x)) , & \text{if }  s \le r  ,\\
\DD_{r-1}(x, Y) \omega_k(f^{(r)},|x-y_{\sigma_r(x)}|^{1/k} \rho_n^{1-1/k}(x)) , &\text{if }  s \ge r+1.
\end{cases}
\end{align}
\end{theorem}

The following lemma follows from, \eg   \cite{KLSUMZh}*{Theorem 5.2 and Lemma 3.1}.

\begin{lemma} \label{umzhlem}
Let $r\in \N_0$ and $s\in\N$ be such that $s\ge r+1$,   and suppose that a set $Y=\{y_j\}_{j=0}^{s-1} \subset [a,b]$ is such that, if  $s\ge r+2$, then
$\Lambda_r(Y) \ge \lambda (b-a)$,
where $0<\lambda \le 1$ (if $s=r+1$, this condition is not needed). If $f\in C^r[a,b]$  then, for all $x\in [a,b]$,
\be \label{umzh1}
|f(x)-L_{s-1}(x; f, Y)| \le c(s,\lambda) (b-a)^r \w_{s-r}(f^{(r)}, b-a, [a,b]) .
\ee
\end{lemma}

\begin{remark}
It immediately follows from \lem{umzhlem} that Theorems~\ref{another} and \ref{191} are valid for all $n \ge \max\{s-1, k+r-1\}$ if the constants $c$ in \ineq{an2} and \ineq{an222} are also allowed to depend on $\lambda$.
\end{remark}

Using \thm{another}, \lem{traux} and \thm{mainnew1} one   immediately arrives at a similar result for simultaneous approximation. However, its statement in the form of \thm{191}  would be rather technical (to estimate the rate of approximation of $f^{(\nu)}$ by the $\nu$-th derivative of $P_n$ we would need to work with a set $Y^\nu$ obtained from $Y$ by removing all points whose multiplicity is at most $\nu$ and reducing multiplicities of all other points by $\nu$). So, instead, we state this result in a simpler (but not as general) form similar to that of
 \thm{mainnew}.

\begin{corollary} \label{19}
Let $k,s \in\N$, $r\in \N_0$, and  let $Y=\{y_j\}_{j=0}^{s-1} \subset I$ be such that $1\le m_j \le r+1$, $0\le j \le s-1$. If $f\in C^r$ then, for every $n \ge N(k, r, s, \delta(Y))$, there exists a polynomial $P_n \in\Pn$ such that, for all $0\le \nu \le r$ and $x\in I$,
\[
|f^{(\nu)}(x)- P_n^{(\nu)}(x)| \le c(k,r,s)   \rho_n^{r-\nu}(x) \w_k(f^{(r)}, \rho_n(x)),
\]
and, moreover,
  for all  $0\le \nu \le r$, $1\le \ell \le k$ and  $0\le j \le s-1$,   if $|x-y_j| \le \rho_n(y_j)$, then
\begin{align*}
| & f^{(\nu)}(x)-P_n^{(\nu)}(x)| \\
& \le c(k,r,s)
\begin{cases}
|x-y_j|^{\sigma_j} \rho_n^{r-\nu-\sigma_j}(y_j)  \omega_\ell(f^{(r)}, \rho_n(y_j)) , & \text{if }\; m_j \le r  ,\\
|x-y_j|^{r-\nu} \omega_\ell(f^{(r)},|x-y_j|^{1/\ell} \rho_n^{1-1/\ell}(y_j)) , & \text{if }\; m_j=r+1 ,
\end{cases}
\end{align*}
where $\sigma_j := \max\{ m_j-\nu, 0\}$.
\end{corollary}

 \begin{remark}
It follows from \lem{umzhlem} that \cor{19} is valid for all $n \ge \max\{s-1, k+r-1\}$ if all constants $c$ are allowed to depend on $\delta(Y)$.
Also,
by allowing all constants in \cor{19}  to depend on $Y$ and setting $\ell=k$ and  $m_j=r+1$, for all $0\le j\le s-1$, we get  \cite{Trigub}*{Theorem 2}.
\end{remark}

It is easy to see that the dependence of the constants $c$ and $N$ on $\lambda$ in \lem{umzhlem} and \thm{another} (and so in \thm{191}), respectively, cannot be removed. For example, if $[a,b]=I$,
$f$ is such that $f^{(r)}(x) = \epsilon^{-1} (x-1+\epsilon)_+$, and $Y$ consists of $s=2r+2$ points $z_0=1-\epsilon$ and $z_1=1$, each with multiplicities $r+1$, then the first modulus of $f^{(r)}$ is bounded above by $1$ and, at the same time, any polynomial $P_n$ whose $r$-th derivative interpolates $f^{(r)}$ at $z_0$ and $z_1$, has to satisfy
$\norm{P_n^{(r+1)}}{} \ge \epsilon^{-1}$, and so, by Markov's inequality, $\norm{P_n}{} \ge \epsilon^{-1} n^{-2r-2}$.
Hence, $P_n$
cannot satisfy \ineq{umzh1}   or \ineq{an2} if constants $c$ and $N$ there do not depend on $\epsilon$.
In particular, this implies that the statement of \cite{Trigub}*{Lemma 3} (even after a correction of a few obvious misprints) is wrong ($\gamma_2$ there cannot be independent of $X_1$).

 The outline of the remaining sections of this paper is as follows. After discussing the history and several  versions of the Dzyadyk-Lebed'-Brudnyi theorem in \sectio{dlb} we use it to provide a simple proof of  \lem{traux} in \sectio{simultsect}. Theorems~\ref{mainnew1} and \ref{another}  are proved in Sections~\ref{sec3} and \ref{thproof}, respectively. \sectio{sec5} is devoted to negative theorems: after proving the  negative result, \thm{negative}, we discuss a much simpler but not as powerful  weak version of this theorem. Finally, several applications are given in \sectio{sec6}.

\sect{Auxiliary statements and proof of \lem{traux}} 

\subsection{Dzyadyk-Lebed'-Brudnyi theorem} \label{dlb}

For each $\alpha>0$ and $M\ge 1$, we denote
\[
\tPhi^\alpha (M)  := \left\{ \phi: (0,\infty) \to (0,\infty)  \suchthat   \phi(t)\uar  \andd  t^{-\alpha}\phi(t) \dar   \right\},
\]
where
\begin{align*}
\psi \uar & \; \iff \; \psi(t_1)\le M \psi(t_2), \; 0<t_1\le t_2 , \andd \\
\psi \dar & \; \iff \; \psi(t_1)\ge   \psi(t_2)/M, \; 0<t_1\le t_2 .
\end{align*}
We also let $\Phi^\alpha := \tPhi^\alpha (1) $, \ie
\[
\Phi^\alpha := \left\{ \phi: (0,\infty) \to (0,\infty)  \suchthat   \phi(t)\uparrow  \andd  t^{-\alpha}\phi(t)\downarrow  \right\},
\]
where we use the notation $g\uparrow$ ($g\downarrow$) to indicate that $g$ is nondecreasing (nonincreasing).
 It is not difficult to see that $\Phi^\alpha \subset C(0,\infty)$ and $\tPhi^\alpha (M) \not\subset C(0,\infty)$ if $M>1$. At the same time, any function from $\tPhi^\alpha (M)$ has the same order of magnitude as a function from
 $\Phi^\alpha$. More precisely, the following lemma is valid.

 \begin{lemma} \label{lem21f}
For any $\alpha >0$, $M\ge 1$ and $\w\in \tPhi^\alpha (M)$, there exists $\w^*\in\Phi^\alpha$ such that $\w(t) \le \w^*(t) \le M^2 \w(t)$, $t>0$.
 \end{lemma}

\begin{proof} Given $\w\in \tPhi^\alpha (M)$, we first define $\widetilde\w (t) := \sup_{0<u\le t} \w(u)$ and note that
$\widetilde\w (t) \uparrow$. Also, $\widetilde\w \in \tPhi(M)$. Indeed, suppose that $0<t_1 < t_2$. If $\widetilde\w (t_1) = \widetilde\w (t_2)$, then
$t_1^{-\alpha} \widetilde\w (t_1) \ge t_2^{-\alpha} \widetilde\w (t_2)$. Otherwise, $\widetilde\w (t_1) < \widetilde\w (t_2)$,  and so
for each $\eps>0$, there is $t_* \in (t_1, t_2]$ such that $\w(t_*) \ge \widetilde\w(t_2) - \eps$, which implies
\[
t_1^{-\alpha} \widetilde\w (t_1) \ge t_1^{-\alpha}\w (t_1)   \ge   t_*^{-\alpha}  \w (t_*)/M  \ge   t_2^{-\alpha}(\widetilde\w (t_2) - \eps)/M ,
\]
and it remains to take $\eps\to 0$ to conclude that $\widetilde\w \in \tPhi(M)$. Note also that $\w(t)\le \widetilde\w(t) \le M\w(t)$, $t>0$.
Now, by Stechkin's theorem (see, \eg \cite{DS}*{p. 202}),  if  $\w^*(t):= t^\alpha \sup_{u>t} \widetilde\w(u)/u^\alpha$, then $\w^*\in\Phi^\alpha$ and
$\widetilde\w(t) \le \w^*(t)\le M \widetilde\w(t)$, $t>0$, and the lemma is proved.
\end{proof}

In particular, it immediately follows from
\lem{lem21f}    that, if $f\in C^r$, then $\phi(t) := t^r \w_k(f^{(r)}, t) \sim \phi^*(t)\in \Phi^{k+r}$.

While it is clear that, if  $\phi(t)\in\Phi^\alpha$, then $t^s \phi(t)\in\Phi^{\alpha+s}$, for any $s\ge 0$, this statement is no longer true if $s<0$. In fact, the following stronger result holds.

\begin{lemma} \label{phi1}
For any $\alpha>0$,  there exists $\phi\in\Phi^\alpha$ such that, for any  $0<\beta<\alpha$, $s\in\R$ and $M\ge 1$,
$t^s \phi(t) \not\in \tPhi^\beta(M)$.
\end{lemma}

\begin{proof} Let   $u(t):= 2/(2-\alpha\ln t)$ and $l_0(t):= t^{\alpha/2}$. Note that both $u$ and $l_0$ are increasing on $[0,1]$ and $l_0(t)<u(t)$, $0<t<1$.
We define the sequence $\{(t_j,\phi_j)\}_{j=0}^\infty$ of points as follows. Starting with $(t_0,\phi_0):=(1/2, l_0(1/2))$,
if a point $(t_{2j},\phi_{2j})$ has been defined,    we pick $t_{2j+1}$ to be  such that
$u(t_{2j+1})=\phi_{2j}$, and  $\phi_{2j+1}:= \phi_{2j}$.
Then, we let $(t_{2j+2},\phi_{2j+2})$ be   the point of intersection of the curves $y=\lambda_j t^{\alpha}$ and $y=l_j(t):= t^{(j+1) \alpha/(j+2)}$, where $\lambda_j$ is chosen so that $\lambda_j t_{2j+1}^{\alpha} = u(t_{2j+1})=\phi_{2j}$.
 It is clear from the construction that $\{t_j\}_{j=0}^\infty$ is a strictly decreasing sequence approaching $0$, and we now   define $\phi$ so that
\[
\phi(t):=\begin{cases}
u(t_{2j+1}),\quad&\text{if}\quad t_{2j+1}\le t < t_{2j},\\
\lambda_j t^{\alpha} ,\quad&\text{if}\quad t_{2j+2}\le t\le t_{2j+1},
 \end{cases}
\]
and $\phi(t):=\phi(t_0)$, $t\ge t_0$. Evidently, $\phi\in\Phi^\alpha$ and $\phi(t_j)=\phi_j$, $j\ge 0$.
 Suppose now that, for some $0<\beta<\alpha$, $s\in\R$ and some $M\ge 1$, $\psi(t)=t^s \phi(t)$ is in $\tPhi^\beta(M)$.
Since $\psi(t_{2j+1})=t_{2j+1}^s u(t_{2j+1})$, $j\in\N$, and $\psi(t)\!\uar$, we must have $s\ge 0$. At the same time,
since $\phi_{2j} = l_{j-1}(t_{2j}) = t_{2j}^{j\alpha/(j+1)}$, $j\ge 1$, we have
$\psi(t_{2j}) 
=    t_{2j}^{s+ j\alpha/(j+1)}$, and so, since $t^{-\beta}\psi(t)\dar$, we must have $t_{2j}^{s+ j\alpha/(j+1)-\beta}  \ge t_{2j_0}^{s+j_0\alpha/(j_0+1)-\beta}/M$, for all $j \ge   j_0 \ge 0$. Since $t_{2j}\to 0$, $j\to \infty$, this yields $s+\alpha-\beta\le 0$.
 Therefore,  $\beta-\alpha  \ge s\ge 0$ which is a contradiction.
\end{proof}

We are now ready to state the following    well known
inequality which is often called in the literature Dzyadyk-Lebed'-Brudnyi inequality.

\begin{lemma}[Dzyadyk-Lebed'-Brudnyi inequality] \label{lebed}
For any $\alpha >0$, $\phi\in\Phi^\alpha$, $s\in\R$,     $n, \nu\in\N$ and  $P_n\in\Pn$, we have
\[
\norm{\rho_n^{s+\nu} P_n^{(\nu)}\phi^{-1}(\rho_n)}{} \le c(\nu, s, \alpha)       \norm{\rho_n^{s} P_n\phi^{-1}(\rho_n)}{} ,
\]
where the constant $c$ may depend only on $\nu$, $s$, and $\alpha$, and  is independent of $n$ and $P_n$.
\end{lemma}

We now give a brief history of \lem{lebed}  which is rather interesting. In 1956, Dzyadyk \cite{D56}*{Theorem $2'$} proved \lem{lebed} with $\phi \equiv 1$. In 1957, Lebed' \cite{L57}*{Theorem 4} (with an obvious misprint: ``$2a/(b-a)$'' in the statement of this theorem should be replaced by ``$2n/(b-a)$'')  established \lem{lebed}  for all norms $L_p$, $1\le p\le \infty$, but only in the case $0<\alpha \le 1$, and it follows from \lem{phi1} that the general case for all $\alpha>0$ cannot be reduced to $0<\alpha \le 1$. In 1959, Brudnyi \cite{B59}*{Theorem $3^*$} stated \lem{lebed} for the first modulus of continuity (\ie  in the case $0<\alpha \le 1$) and $(-s)\in\N_0$. It seems that he was not aware of the work of Lebed' despite the fact that, at that time, both were living and working in Dnipropetrovs'k (currently Dnipro), Ukraine.

In 1959, Dzyadyk \cite{D59} established the following result which, for the interval $I$, can be restated as follows.

\begin{theorem}[see \cite{D59}*{Theorem 3.3}] \label{dz59}
Suppose that a positive continuous function $A$ is defined at all points of $I$, and suppose that a polynomial $P_n\in\Pn$ satisfies
\[
|P_n(x)|\le A(x) , \quad x \in I.
\]
Then, for all $x\in I$,
\be \label{dz2}
|P_n^{(k)}(x)|\le(1+\varepsilon)e k!\frac{ A(x)}{\rho_n^k(x)},
\ee
where $\varepsilon$ uniformly tends to $0$ as $n\to \infty$.
\end{theorem}

We emphasize that $\varepsilon$ in \ineq{dz2} is only guaranteed to be bounded by an absolute constant for sufficiently large $n$ (depending on the function $A$), and, for any $E>0$, one may find $n\in\N$ and a function $A$ such that $\varepsilon$ in   \ineq{dz2} is not smaller than $E$.
Indeed, if $T_n(x) = \cos(n \arccos  x)$ is a Chebyshev polynomial and $n\in\N$ is odd, then the polynomial
 $P_n(x):=T_n(nx)$ satisfies $|P_n(x)|\le A(x):=\max\{1,|T_n(x)|\}$.
At the same time, $|P_n'(0)|=n^2 > n \rho_n^{-1}(0) A(0)$, and so $\varepsilon > n/e-1$.

 Hence, one may not apply \thm{dz59}   if $A$ is allowed to depend on $n$. Nevertheless, the same idea as was used in the proof of this theorem in \cite{D59} yields the following result.
 (For its alternative proof   see \cite{KLS-cjm}*{Lemma 5.2}.)

\begin{lemma}[Dzyadyk inequality, see  \cite{DS}*{p. 386}] \label{88}
Suppose that $m\in\N$  and $x_0\in I$. If
\[
|P_n(x)|\le(|x-x_0|+\rho_n(x))^m,\quad x\in I,
\]
then
\[
|P_n'(x_0)|\le c(m) \rho_n^{m-1}(x_0) .
\]
\end{lemma}

\lem{88} is rather powerful and can be used to establish various general results involving uniform norms of polynomials and their derivatives. In particular, it can be used to almost immediately obtain \lem{lebed} as stated (see \cite{DS}*{pp. 383, 387}).

We remark that \lem{lebed} immediately implies what seems to be a stronger result.

\begin{corollary}  \label{dzstrong}
For any $\alpha, \co >0$, $\phi\in\Phi^\alpha$, $s, \mu\in\R$,     $n, \nu\in\N$ and  $P_n\in\Pn$, we have
\be \label{strin}
\norm{\rho_n^{s+\nu} P_n^{(\nu)}\phi^{-1}(\co \rho_n^{\mu})}{} \le c(\nu, s, \alpha,\mu)       \norm{\rho_n^{s} P_n\phi^{-1}(\co\rho_n^{\mu})}{} ,
\ee
where the constant $c$ may depend only on $\nu$, $s$, $\alpha$ and $\mu$, and   is independent of $n$, $P_n$ and $\co$.
\end{corollary}

\begin{proof} First, note that if $\mu=0$, then \ineq{strin} is the classical Dzyadyk inequality \cite{D56}*{Theorem $2'$}.
If $\mu\neq 0$, we define
\[
\psi(t):=
\begin{cases}
\phi(\co t^\mu) , & \text{if }\; \mu>0 , \\
t^{-\alpha \mu} \phi(\co t^\mu)  , & \text{if }\; \mu<0 ,
\end{cases}
\]
and  it is straightforward to check that $\psi \in \Phi^{\alpha |\mu|}$. It remains to use \lem{lebed} with $\phi$ replaced by $\psi$.
\end{proof}

We note that \cor{dzstrong} (with $\co = n^{-\lambda}$ and $\mu=1-\lambda$)
 is an improvement of the following result by Ditzian and Jiang  \cite{DJ}*{Theorem 4.1}

\begin{lemma}[  \cite{DJ}*{Theorem 4.1}] 
Let $\alpha>0$, $n\in\N$ and  $\phi\in\Phi^\alpha$. Then, for every $P_n\in\Pn$,  $s\in\R$,  $0\le\lambda\le 1$  and $\nu \ge -s+\alpha(1-\lambda)$, we have
\[
\norm{ \rho_n^{s+\nu}  P_n^{(\nu)} \phi^{-1}\left(n^{-\lambda}\rho_n^{1-\lambda} \right) }{} \le c(l,s,\alpha,\lambda) \norm{ \rho_n^{s} P_n  \phi^{-1}\left(n^{-\lambda}\rho_n^{1-\lambda} \right) }{}.
\]
\end{lemma}

\subsection{Proof of \lem{traux}} \label{simultsect}

It is well known (see,  \eg  \cites{Tr62} and \cite{DS}*{Theorem 7.3.3}) that, for any $n\ge k+r-1$, there exists a polynomial $Q_n\in\Pn$ such that
estimates \ineq{tr1} and \ineq{tr2} (with $A=1$ and $P_n$ replaced by $Q_n$) are both satisfied for all $x\in I$.
Suppose now that  $P_n\in\Pn$ satisfies \ineq{classKK}   and denote $R_n:= Q_n-P_n$.
Then,
\[
|R_n(x)| \le c A \rho_n^r(x)   \w_k(f^{(r)}, \rho_n(x) ),  \quad x\in I,
\]
and applying \lem{lebed} with $s=0$, $\alpha =k+r$ and  $ t^r \w_k(f^{(r)}, t)  \sim \phi(t)\in\Phi^{k+r}$,  we conclude that, for any $\nu\in\N$,                 
\[
|R_n^{(\nu)}(x)| \le c A \rho_n^{r-\nu}(x) \w_k(f^{(r)}, \rho_n(x))    , \quad x\in I.
\]
Hence, for $x\in I$,
\begin{align*}
|f^{(\nu)}(x)-P_n^{(\nu)}(x)| & \le |f^{(\nu)}(x)-Q_n^{(\nu)}(x)| + |R_n^{(\nu)}(x)|  \\
& \le c A \rho_n^{r-\nu}(x) \w_k(f^{(r)}, \rho_n(x))  , \quad   0\le \nu\le r,
\end{align*}
and
\[
|P_n^{(k+r)}(x)| \le |Q_n^{(k+r)}(x)| + |R_n^{(k+r)}(x)| \le c A \rho_n^{-k}(x) \w_{k } (f^{(r)}, \rho_n(x)) ,
\]
and the lemma is proved.

\sect{Proof of \thm{mainnew1}} \label{sec3}

We start by recalling the definition  of the  Lagrange-Hermite divided difference of $f$ of order $m$ at the knots $Y = \{y_j\}_{j=0}^m$ for which we use the notation
 $[y_0,\dots,y_m;f]$ (see, \eg  \cite{DL}*{Section 4.7} or \cite{DS}*{Section 3.8.3}).
Given $m\in\N_0$, if
  $y_0=\dots=y_m$, then $[y_0, \dots, y_m; f] = f^{(m)}(y_0)/m!$.
Otherwise,   $y_0\ne y_{j^*}$, for some  $j^*$, and
\[
[y_0,\dots,y_m;f]:=\frac1{y_{j^*}-y_0} \left([y_1,\dots,y_{m};f]-[y_0,\dots,y_{j^*-1},y_{j^*+1},\dots,y_m;f]\right).
\]
Recall that   $[y_0,\dots,y_m;f]$ is symmetric in $y_0, \dots, y_m$ (\ie it does not depend on how the points from $Y$ are numbered).
Then the Lagrange-Hermite polynomial $L_m(\cdot; f, Y)$ of degree $\le m$ that satisfies
\[
L_m^{(l_j-1)}(y_j; f, Y) = f^{(l_j-1)}(y_j) , \quad \text{for all }\;   0\le j\le m ,
\]
may be written as
\[
L_m(x; f, Y)   := f(y_0)+  \sum_{j=1}^m [y_0, \dots, y_j; f](x-y_0)\dots (x-y_{j-1}) .
\]
In particular, this implies
\be \label{new}
f(x)-L_m(x; f, Y) = [x, y_0, \dots, y_m; f] \prod_{j=0}^m (x-y_j) , \quad \text{for }\; x\not\in Y.
\ee
The main property of divided differences that we need in this section is that, if all $y_j$'s lie inside some interval $J$ and $f\in C^m(J)$, then $[y_0, \dots, y_m; f] = f^{(m)}(\theta)/m!$, for some $\theta\in J$.

We   now prove the following lemma and then show that \thm{mainnew1} immediately follows from it.

\begin{lemma} 
Let $k\in\N$,  $r,m,n\in\N_0$,   $m\le r$,  $f\in C^r$, $P_n\in\KK$,   $Y=\{y_j\}_{j=0}^{m} \subset I$
and $x\in I$ be given. If, for all $0\le j \le m$,
 \be\label{3333}
P_{n}^{(l_{j}-1)}(y_{j}) = f^{(l_j-1)}(y_{j}),
\ee
and
\[
|x-y_j|\le|x-y_m|\le\rho_n(x),
\]
 then
 \begin{align} \label{783}
| & f(x)-P_n(x)| \nonumber \\
& \le c(k,r)A
\begin{cases}
|p_m(x)|\rho_n^{r-m-1}(x)  \omega_k(f^{(r)}, \rho_n(x)) , & \text{if }\; m \le r-1 ,\\
|p_{r-1}(x)| \omega_k(f^{(r)}, |x-y_r|^{1/k}  \rho_n^{1-1/k}(x)) , &\text{if }\; m=r,
\end{cases}
\end{align}
where
\[
p_m(x):=\prod_{j=0}^m(x-y_j).
\]
\end{lemma}

\begin{proof}  Throughout the proof, it is convenient to denote $\rho := \rho_n(x)$,   $\uw(t) := \w_k(f^{(r)}, t)$ and
  $J_n := [x -\rho , x +\rho ]\cap I$.
  Note that $\rho  \sim \rho_n(\theta)$, for every $\theta\in J_n$.

If $x\in Y$, there is nothing to prove, so we assume that $x\notin Y$. Put $g:=f-P_n$ and note that   \ineq{new} implies
\be\label{334}
g(x)=[x,y_0\dots,y_m;g]p_m(x),
\ee
 since \ineq{3333} implies that $L_{m}(\cdot; g, Y)\equiv 0$.

If $m\le r-1$, then $g\in C^{m+1}$, and so
there is a point $\theta\in J_n$, such that
$[x,y_0\dots,y_m;g]=g^{(m+1)} (\theta)/(m+1)!$,
 and it follows from \ineq{334} and  \ineq{tr1} with $\nu=m+1$  that
\begin{align*}
|g(x)|&=\frac{|g^{(m+1)} (\theta)|}{(m+1)!}|p_m(x)|\le cA|p_m(x)| \rho_n^{r-m-1}(\theta) \uw(\rho_n(\theta)) \\
& \le cA|p_m(x)| \rho^{r-m-1} \uw(\rho).
\end{align*}
Hence, \ineq{783} is proved for $m\le r-1$.

If $m=r$,
denoting by $J$ the smallest interval   containing $x$ and all $y_j$'s and using the fact that
 $g\in C^r$, we conclude that there are $\theta_1, \theta_2\in J$  such that
\[
[x,y_0,\dots,y_r;g]=\frac{[x,y_0,\dots,y_{r-1};g]-[y_0,\dots,y_{r};g]}{x-y_r}=\frac1{r!}\frac{g^{(r)}(\theta_1)-g^{(r)}(\theta_2)}{x-y_r}.
\]
Since $|\theta_1- \theta_2| \le |J| \le2|x-y_r|$, together with \ineq{334},  this implies
\begin{align} \label{nice}
|f(x)-P_n(x)| & =\frac1{r!}|g^{(r)}(\theta_1)-g^{(r)}(\theta_2)|\frac{|p_r(x)|}{|x-y_r|} \nonumber \\
& \le  c|p_{r-1}(x)|\omega_1(g^{(r)},|x-y_r|;J_n).
\end{align}
Now,   estimate
\ineq{tr2}  yields, for any $0<t\le  2\rho$ and   some $\theta\in J_n$,
\begin{align*}
\omega_k(P_n^{(r)},t;J_n) &\le  c t^k|P_n^{(r+k)}(\theta)|\le cAt^k\frac{\omega_k(f^{(r)},\rho_n(\theta))}{\rho_n^k(\theta)}\le cA\uw(t).
\end{align*}
Hence,
\be \label{22}
\omega_k(g^{(r)},t;J_n)\le \uw(t)+\omega_k(P_n^{(r)},t; J_n)\le cA\uw(t), \quad 0<t\le  2\rho,
\ee
and letting $t := |x-y_r|$ and using   \ineq{nice}, we obtain  \ineq{783} in the case $k=1$ and $m=r$.

If $m=r$ and $k\ge 2$, we use the well known Marchaud inequality (see,  \eg  \cite{DL}*{Theorem 2.8.1})
\[
\omega_1(g^{(r)},t ; J_n )\le ct\int_t^{|J_n|}\frac{\omega_k(g^{(r)},u;J_n )}{u^2}\, du + c t |J_n|^{-1}   \norm{g^{(r)}}{J_n } , \quad 0<t\le  \rho .
\]
Estimate \ineq{tr1} with $\nu=r$ implies that  $\norm{g^{(r)}}{J_n } \le cA \norm{\uw(\rho_n(\cdot))}{J_n} \le cA \uw(\rho)$.
Hence, if  $0<t<  \rho $ and   $\eta\in[t, \rho]$,  then  applying
  \ineq{22} and the inequality $u_2^{-k} \uw(u_2) \le 2^k u_1^{-k} \uw(u_1)$, $0<u_1 < u_2$, and using $\rho\le |J_n|\le 2\rho$,  we get
\begin{align*}
A^{-1}  \omega_1(g^{(r)},t ;J_n )&\le  ct\int_t^{2\rho}\frac{ \uw(u)}{u^2}\, du +
c t \rho^{-1}  \uw(\rho)\\
&\le ct\left(\int_t^{\eta}+\int_{\eta}^{\rho}\right)   \frac{\uw(u)}{u^2}\, du +
c t \rho^{-1}\uw(\rho)\\
& \le ct \uw(\eta)\int_t^{\infty} u^{-2} \, du +  ct \eta^{-k} \uw(\eta) \int_0^{\rho}  u^{k-2} \, du +  c t \rho^{k-1}    \eta^{-k} \uw(\eta)      \\
&\le c\uw(\eta)\left(1+  t \rho^{k-1} \eta^{-k} \right).
\end{align*}
Hence, for   $t = |x-y_r|$ and $\eta = |x-y_r|^{1/k} \rho^{1-1/k}$, we have
\[
   \omega_1(f^{(r)}-P_n^{(r)}, |x-y_r| ;J_n ) \le c A \uw\left( |x-y_r|^{1/k} \rho^{1-1/k}\right) ,
\]
which combined with \ineq{nice} implies \ineq{783} in the case $k\ge 2$ and $m=r$.
 \end{proof}

\begin{proof}[Proof of  \thm{mainnew1}]
First, since the estimate \ineq{78} in the case $s\ge r+1$ depends only on $r+1$ points from $Y$ which are closest to $x$, without loss of generality, we may assume that $s\le r+1$.
Now, let   $n\in\N_0$, $1\le s\le r+1$,  $Y =\{y_j\}_{j=0}^{s-1}$ and $x\in I$ be given. If $[x-\rho_n(x), x+\rho_n(x)] \cap Y = \emptyset$, then \ineq{78} follows from \ineq{classKK}.
Otherwise,
let $m\in\N_0$ be the largest number $\le s-1$  such that $|x-y_{\sigma_m(x)}|\le\rho_n(x)$.
Then, either (i) $m=s-1$ and $Y \subset [x-\rho_n(x), x+\rho_n(x)]$, or (ii) $m\le s-2$ and $|x-y_{\sigma_{m+1}(x)}| > \rho_n(x)$.
In the case (i), estimates \ineq{78} and \ineq{783} are identical (with an obvious change if $y_0$ is the farthest point from $x$ instead of $y_m$). In the case (ii), $m\le r-1$ and $\DD_{s-1}(x, Y) \ge  \DD_{m}(x, Y) \rho_n^{s-m-1}(x)$.
Therefore, \ineq{78} follows from \ineq{783} with $m\le r-1$ taking into account that, if $s=r+1$, then $|x-y_{\sigma_r(x)}| \ge |x-y_{\sigma_{m+1}(x)}| > \rho_n(x)$.
\end{proof}

\sect{Proof of \thm{another}} \label{thproof}


Given $f\in C^r$ and a set $Y = \{y_j\}_{j=0}^{s-1} \subset I$, we first use \lem{umzhlem} to construct a piecewise polynomial function $S$ having the right local order of approximation and Hermite interpolating $f$ at the points in $Y$. We then approximate $S$ by a polynomial satisfying all conditions of \thm{another}.

Let
$ x_{j}:= x_{j,n} := \cos(j\pi/n)$, $0 \leq j \leq n$, denote the Chebyshev nodes, and let $I_j := I_{j,n} := [x_{j}, x_{j-1}]$,
\[
\psi_{j}(x) := \psi_{j,n}(x) := \frac{ |I_j|}{|x-x_j|+|I_j|} \andd
\chi_j(x) :=\chi_{j,n} (x) :=
\begin{cases}
1 , & \text{if }\;   x\geq x_j , \\
0,  & \text{otherwise.}
\end{cases}
\]
 where $1 \leq j\leq n$.

We start with the following lemma which is an immediate consequence of \cite{hky}*{Corollary 17}.

\begin{lemma}  \label{h15}
Let   $n,\mu\in\N$,   $1\leq j\leq n-1$,
and let the numbers
$a$ and $b$ be such that $-1\leq a \leq x_{j+1}$ and $x_{j-1}\leq b\leq 1$.
Then,
 there exists a polynomial
$T_j(x) := T_j(x; a, b)$ of degree
$\leq c(\mu)n$ such that
\[ 
T_j (a) = \chi_j(a) = 0 , \quad T_j (b) = \chi_j(b) = 1 ,
\]
 and
\[ 
|\chi_j(x)- T_j (x)| \leq c(\mu) \psi_j^\mu(x), \quad x\in I.
\]
\end{lemma}

\begin{corollary} \label{newcorol}
Let  $n,s,\mu\in\N$, $r\in\N_0$,
  $Y = \{y_j\}_{j=0}^{s-1} \subset I$, and let index $1\le j \le n-1$ be such that $(x_{j+1}, x_{j})$ does not contain any points from $Y$.
  Then,
   there exists a polynomial
$R_j(x):=R_{j,n}(x):= R_j(x; Y)$ of degree $\leq c(\mu,r,s)n$ such that
\be \label{hh11}
R_j(y_i) = \chi_j(y_i), \quad R_j^{(\nu)} (y_i) =0 ,\; 1\le \nu \le r , \quad 0\le i \le s-1,
\ee
and
\be \label{h22}
|\chi_j(x)- R_j (x)| \leq c(r,s,\mu) \psi_j^\mu(x), \quad x\in I.
\ee
\end{corollary}


\begin{proof} First, we may assume that the interval $(x_{j+1}, x_{j-1})$ does not contain any points from $Y$.
To see this, it is sufficient to replace $n$ by $2n$ and, denoting $\tj := 2j+1$, notice that $(x_{j+1}, x_j) = (x_{\tj+1, 2n}, x_{\tj-1, 2n})$,
  $|\chi_j(x)-\chi_{\tj, 2n}(x)| \le c \psi_j^\mu(x)$ and $\psi_{\tj, 2n} (x)\sim \psi_{j}(x)$, $1\le j\le n-1$, and so if $R_{\tj, 2n}$ satisfies \ineq{hh11} and \ineq{h22} with $j$ and $n$ replaced by $\tj$ and $2n$, respectively,   then it also satisfies these inequalities as stated.

Also, it is clear that, without loss of generality, we may assume that $\pm 1 \in Y$.
Now, for convenience, we denote  $A = \{a_1, \dots, a_\ell\} : = Y \cap [-1, x_{j+1}]$ and $B = \{b_1, \dots, b_m\} : = Y \cap [x_{j-1}, 1]$, where $\ell, m \ge 1$ and $\ell + m = s$.
Let
\[
Q_j(x; A, b_k) := \prod_{i=1}^\ell  T_j^{r+1}(x; a_i, b_k) , \quad 1\le k \le m ,
\]
where $T_j$ are polynomials from \lem{h15}.
It is clear that $Q_j(b_k; A, b_k) = 1$, $Q_j^{(\nu)}(a_i; A, b_k) = 0$, for $0\le i\le \ell$ and $0\le \nu\le r$, and $Q_j$ approximates $\chi_j$ with the right order (\ie \ineq{h22} holds with $Q_j$ instead of $R_j$).
Now, let
\[
R_j(x; Y) :=   1-\prod_{k=1}^m \left( 1- Q_j(x; A, b_k) \right)^{r+1} .
\]
Then $R_j(a_i; Y) = 0$, $1\le i \le \ell$, $R_j(b_i; Y) = 1$, $1\le i \le m$, and $R_j^{(\nu)}(y; Y) =0$, for all $y\in Y$ and $1\le \nu\le r$. Also, it is not difficult to check that \ineq{h22} holds.
\end{proof}

\begin{remark}
We note that \cor{newcorol} can also be proved using the method from \cite{GS}.
\end{remark}

We are now ready to prove \thm{another}. It is   clear that, by increasing $k$ if necessary, we may assume that $k+r \ge s$, and by replacing $f$ with $g = f-L_{s-1}(\cdot; f, Y)$ ($L_{s-1} \in \Poly_{s-1}$ is the polynomial satisfying \ineq{hermite})  and noting that $\w_k(g^{(r)}, t) = \w_k(f^{(r)}, t)$, we may assume that
\[
f^{(l_j-1)}(y_j)=0 , \text{for all }\;   0\le j \le s-1.
\]

%
Suppose now that $n\in\N$ is so large that
any set consisting of $2r+3$ adjacent intervals $I_i$ contains at most $r+1$ points from $Y$. (This is achieved if the total length of these intervals is less than $\lambda$, and so $n\ge 50 (r+1)/\lambda$ will do). Since each point $y_j$ from $Y$ can belong to at most two intervals $I_i$, this also means that this set contains at least one interval $I_i$ that does not contain any points from $Y$.

For each $0\le j \le s-1$, let $O_j$ be the smallest interval $J = [x_{i_2}, x_{i_1}]$ such that $y_j \in J$ and (i) either $x_{i_2}=-1$ or  $[x_{i_2 +1}, x_{i_2}]$  does not contain any points from $Y$,   and
(ii) either $x_{i_1}=1$ or  $[x_{i_1}, x_{i_1-1}]$   does not contain any points from $Y$.
 Some observations are in order: (i) each interval $O_j$ consists of not  more than $2r+2$ intervals $I_i$ and contains at most $r+1$ points from $Y$, (ii) a point  from $Y$ cannot be an endpoint  of $O_j$ unless it is the endpoint of $I$,
 (iii) it is possible for $Q_i$ and $Q_j$ to be the same even when $i\ne j$ (and, in fact, even in the case $y_i \ne y_j$).

 We now denote $O := \cup_{j=0}^{s-1} O_j$ and  construct a piecewise polynomial function $S$ as follows. On the intervals $I_i$ that are not contained in $O$, we let $S$ to be any polynomial of degree $\le k+r-1$ satisfying
\be \label{local}
\norm{f(x)-S|_{I_i}}{I_i}  \le c (k,r) |I_i|^r \w_k(f^{(r)}, |I_i|, I_i) .
\ee
Such a polynomial is guaranteed by Whitney's theorem.

For all $i$ such that $I_i \subset O_j$, for some  $0\le j \le s-1$,  we let $S$ be the polynomial $L_{k+r-1}(\cdot; f, \widetilde Y)$ from \lem{umzhlem} where  $[a,b] := O_j$, and $\widetilde Y$ consists of all points in $Y \cap O_j$ supplemented by however many extra points in $O_j$ are needed so that the cardinality of $\widetilde Y$ becomes $k+r$ (while keeping $\lambda$ from the statement of \lem{umzhlem} bounded below by a positive constant that depends only on $k$ and $r$). Then $S|_{O_j}  \in\Poly_{k+r-1}$ satisfies \ineq{local} with $I_i$ replaced by $O_j$,
and
$(S|_{O_j})^{(l_i-1)}(y_i)=0$, for all $y_i \in O_j$.

We note that $S$ does not have to be continuous, and so our definition is inaccurate at the endpoints of $I_i$'s since, potentially, $S|_{I_{i+1}}(x_i)$ is not the same as $S|_{I_{i}}(x_i)$. Even though it has no influence on our estimates, to be precise, we now redefine $S$ at all points $x_i$ so that it becomes right continuous everywhere in $I$.

For convenience, denote $p_i := S|_{I_i}$ and note that
\[
S(x) = p_n(x) + \sum_{i=1}^{n-1} \left( p_i(x)-p_{i+1}(x) \right) \chi_i(x) = p_n(x) + \sum_{i\in \Lambda}  \left( p_i(x)-p_{i+1}(x) \right) \chi_i(x) ,
\]
where $\Lambda := \left\{ i \suchthat 1\le i \le n-1, \; \text{and }\; x_i \not\in O^\circ \right\}$.
Finally, we let
\[
P_n(x):= p_n(x) + \sum_{i\in\Lambda}  \left( p_i(x)-p_{i+1}(x) \right) R_i(x; Y) .
\]
Now, using a rather standard approach (see,  \eg  \cites{GS, hky}), one can show that \ineq{an2} holds,
and it only remains to verify that
\be  \label{trrr}
P_{n}^{(l_j-1)}(y_j) = 0, \quad \text{for all }\; 0\le j \le s-1 .
\ee
Indeed, for any $0\le \nu \le r$, we have
\begin{align*}
P_n^{(\nu)}(y_j) &= p_n^{(\nu)}(y_j) + \sum_{i\in\Lambda} \sum_{\ell=0}^\nu \binom{\nu}{\ell}  \left( p_i^{(\nu-\ell)}(y_j)-p_{i+1}^{(\nu-\ell)}(y_j) \right) R_i^{(\ell)}(y_j; Y)   \\
&= p_n^{(\nu)}(y_j) + \sum_{i\in\Lambda}  \left( p_i^{(\nu)}(y_j)-p_{i+1}^{(\nu)}(y_j) \right) \chi_i(y_j) 
= S^{(\nu)}(y_j) ,
\end{align*}
and so \ineq{trrr} follows.


\sect{Negative theorems} \label{sec5}

\subsection{Strong negative theorem: proof of \thm{negative}}

Without loss of generality, we can assume that $x_0$ in the statement of \thm{negative} is $0$. Moreover, as is well known (see,  \eg  \cite{DS}*{Chapter 4}),   any function can be extended from $[0,1]$ to $[-1,1]$ without essentially changing its modulus of smoothness or modulus of smoothness of its derivative. Hence, it is sufficient to prove 
 the following lemma.

\begin{lemma}\label{negative1}
Let $k\in\N$,  $r\in\N_0$,   and let a positive function $\vare \in C(0,1]$ be such that $\lim_{x\to 0^+} \vare(x)=0$. Then,
there is a function $F\in C^{r}[0,1] $, such that for any algebraic polynomial $P$ we have
\be\label{01}
\limsup_{x\to 0^+}\frac{|F(x)-P(x)|}{\omega_k\left(F^{(r)},\vare(x) x^{(r+1)/k} ; [0,1]\right) }=\infty.
\ee
\end{lemma}

\lem{negative1} will be proved in three cases: (i) $k=1$, (ii) $k \ge \max\{2, r+1\}$, and (iii) $2\le k \le r$.

\mbox{}\\
{\bf Case (i): $k=1$.} \\
Let    $F(x) := x^{r+1} \cos(2\pi \ln x)$. Then, $F\in C^r[0,1]$,  $\norm{F^{(r+1)}}{L_\infty[0,1]} \le c(r)$, and so $\w_1(F^{(r)}, t; [0,1]) \le c(r) t$.
If $P$ is an arbitrary polynomial from $\Pn$, then it
  is either identically zero or has at most $n$ zeros, and so there exists $\delta>0$ such that
$P$ is either nonnegative or nonpositive on $[0,\delta]$.
Without loss of generality, suppose that $P(x) \le 0$, $0\le x \le  \delta$.
Then
\[
|F(x_m)-P(x_m)| \ge  F(x_m) = x_m^{r+1} , \quad \mbox{\rm for all }\;  x_m :=  e^{-m},
\]
  where $m\in\N$, $m\ge |\ln \delta|$. Hence,
\[
\frac{|F(x_m)-P(x_m)|}{\omega_1\left(F^{(r)},\vare(x_m) x_m^{r+1} ; [0,1]\right)} \ge \frac{c(r)}{\vare(x_m)} \to \infty , \quad m\to\infty.
\]

\mbox{}\\
{\bf Case (ii): $k \ge \max\{2, r+1\}$.} \\
By replacing  $\vare$ by $\hat\vare(x):=\max \{\vare(x),2 x^{1/k}\}$ if necessary, we may assume that
\be \label{1}
\vare(x)\ge2x^{1/k}\ge 2x .
\ee

We define the sequence $\{x_j\}_{j=0}^\infty$ of points in $(0,1]$ as follows. Starting with $x_0:=1$,
if a point $x_{2j}$ has been defined, then we pick $x_{2j+1}$ to be the smallest number such that
$\vare(x_{2j+1})=x_{2j}$, and let
\[
\wx_{2j+1} := x_{2j+1}\vare(x_{2j+1}) = x_{2j} x_{2j+1} \andd x_{2j+2}:=x_{2j}^{1-k}\wx_{2j+1}^k = x_{2j} x_{2j+1}^k .
\]
  Note that $2x_{2j+2} \le  \wx_{2j+1} \le  x_{2j+1}\le x_{2j}/2$ and
define
\[ 
\omega(x):=\begin{cases} x_{2j+2},\quad&\text{if}\quad x_{2j+2}\le x < \wx_{2j+1},\\
x_{2j}^{1-k}x^k,\quad&\text{if}\quad \wx_{2j+1}\le x\le x_{2j}.
 \end{cases}
\]
Clearly,  
$0<\omega(x)\le x$, $0<x\le 1$, and $\omega\in\Phi^k$.
Therefore, setting
\[
f(x):=\frac 1{(k-2)!}\int_x^1x(u-x)^{k-2}\frac{\omega(u)}{u^k}\, du, \quad x\in[0,1],
\]
it follows from \cite{DS}*{Theorem 3.4.2} that
\be\label{22ss}
\omega(t)\le\w_k(F^{(r)},t; [0,1]) = \w_k(f,t; [0,1])\le k\omega(t),\quad t\in[0,1],
\ee
where
\[ 
F(x):=\begin{cases}
  \frac1{(r-1)!}\int_0^x(x-t)^{r-1}f(t)\,dt , &   \quad \text{if $r\ge1$,}\\
f(x),& \quad   \text{if $r=0$.}
\end{cases}
\]
It is not difficult to show (see also \cite{DS}*{Lemma 3.4.2(i)}) that $\lim_{x\to 0^+} f(x) = f(0)=0$.
Hence, $f\in C[0,1]$ and so, in particular, for $r\ge 1$,
\begin{align} \label{11}
\frac{F(x)}{x^{r}} & \le \big(\max_{t\in[0,x]}f(t)\big) \frac1{(r-1)!x^r}\int_0^x(x-t)^{r-1}\,dt \nonumber \\
& =\frac1{r!}\max_{t\in[0,x]}f(t)\to0,\quad x\to0.
\end{align}
Also, since  for $j\in \N$,
\[ 
\int_{x_{2j}}^{2x_{2j}} (u-x_{2j})^{k-2}  \frac{\omega(u)}{u^k} \,du={x_{2j}}\int_{x_{2j}}^{2x_{2j}}\frac{(u-x_{2j})^{k-2}}{u^k}\,du
=\int_1^2\frac{(u-1)^{k-2}}{u^k}\, du=:c_k,
\]
we have
\begin{align*}
\frac{f(x)}{x} & \ge \frac{1}{(k-2)!} \sum_{j\in\N, x_{2j} \ge x} \int_{x_{2j}}^{2x_{2j}} (u-x)^{k-2}  \frac{\omega(u)}{u^k} \,du \\
& \ge \frac{c_k}{(k-2)!}\sum_{j\in\N, x_{2j} \ge x} 1 \to \infty, \quad x\to 0 .
\end{align*}
 Hence, since $f(x)/x \downarrow$ on $(0,1]$, we have, for $r\ge 1$,
\begin{align}\label{12}
\frac{F(x)}{x^{r+1}}&=\frac1{(r-1)!x^{r+1}}\int_0^x(x-t)^{r-1}t\frac{f(t)}t\,dt
  \ge\frac{f(x)}{(r-1)!x^{r+2}}\int_0^x(x-t)^{r-1}t\,dt \nonumber \\
&=\frac1{(r+1)!}\frac{f(x)}x\to\infty,\quad x\to0.
\end{align}
Also, since by \ineq{1},  $\wx_{2j+1} \le x_{2j+1}^{(r+1)/k} \le x_{2j}$, we have
\begin{align}\label{41}
\omega\left(x_{2j+1}^{(r+1)/k}\right) &=  x_{2j}^{1-k} x_{2j+1}^{r+1} \le x_{2j+1}^{(1-k)/k} x_{2j+1}^{r+1} = x_{2j+1}^{r+1/k}
 =o(x_{2j+1}^{r}),\quad j\to\infty.
\end{align}
Now, we fix a polynomial $P$. If
\[
\limsup_{x\to 0}\frac{|F(x)-P(x)|}{\omega(x^{(r+1)/k})}=\infty,
\]
then \ineq{22ss} implies \ineq{01}.
Otherwise, there is a constant $C$, such that
\[ 
|F(x)-P(x)|\le C\omega(x^{(r+1)/k}),\quad x\in[0,1].
\]
Since, by \ineq{41}  and \ineq{11}, $\omega(x_{2j+1}^{(r+1)/k})=o(x_{2j+1}^r)$ and $F(x_{2j+1})=o(x_{2j+1}^r)$  as $j\to\infty$, we conclude that $P(x_{2j+1})=o(x_{2j+1}^r)$  as well.
This means, that $P(x)=x^{r+1}Q(x)$, where $Q$ is also an algebraic polynomial.
By \ineq{12},
there is   $\delta >0$, such that $\norm{Q}{[0,\delta]} <\frac12 x^{-r-1}F(x)$, $x\in (0,\delta]$, which, in turn, implies
\[
|P(x)| \le \frac12 F(x),\quad x\in[0,\delta].
\]
Denoting, for convenience, $y:=x_{2j+1}$,  it follows from \ineq{12} that, for sufficiently large $j$,
\[
 F(y)-P(y)   \ge \frac 12  F(y) \ge y^{r+1}.
\]
On the other hand, since
$\wx_{2j+1} = y\vare(y) \le y^{(r+1)/k}\vare(y)\le \vare(y)=x_{2j}$, we have $\omega(y^{(r+1)/k}\vare(y))=x_{2j}^{1-k}y^{r+1}\vare^k(y) = \vare(y) y^{r+1}$, and therefore
\[
\frac{F(y)-P({y})}{\omega(y^{(r+1)/k}\vare(y))}\ge\frac {1}{\vare(y)},
\]
which together with \ineq{22ss} implies \ineq{01}.

\mbox{}\\
{\bf Case (iii): $2\le k\le r$.} \\
We will show that, in this case, the statement of the lemma follows from the case $k=r+1$ that was proved above and the Marchaud inequality.
Let
\[
\we (x):=\vare^{k/(r+1)}(x) .
\]
Then, it follows from Case~(ii) that there exists a function $F\in C^r[0,1]$ such that
\be\label{87}
\limsup_{x\to 0}\frac{|F(x)-P(x)|}{\w_{r+1}\left(F^{(r)},x\we(x); [0,1]\right)}=\infty ,
\ee
for every polynomial $P$.
It follows from the Marchaud inequality (recall that it can be found in  \cite{DL}*{Theorem 2.8.1}, for example) that, for  $k\le r$ and $f:=F^{(r)}$, we have, for sufficiently small $x>0$,
\begin{align*}
\omega_k\left(f,  \vare(x)x^{(r+1)/k} ; [0,1]\right)
 \le \; & c\vare^k(x)x^{r+1}\int_{\vare(x)x^{(r+1)/k}}^1\frac{\omega_{r+1}(f, u; [0,1])}{u^{k+1}}\, du \\
 & +c\vare^k(x)x^{r+1} \norm{f}{[0,1]} .
\end{align*}
Denoting $\W(x):= \w_{r+1}\left(f,x\we(x); [0,1]\right)$, we have
\begin{align*}
  \int_{\vare(x)x^{(r+1)/k}}^1\frac{\omega_{r+1}(f,u; [0,1])}{u^{k+1}}\, du
& =  \left( \int_{\vare(x)x^{(r+1)/k}}^{x\we(x)} + \int_{x\we(x)}^{1} \right) \frac{\omega_{r+1}(f,u; [0,1])}{u^{k+1}}\, du   \\
&\le
\left(  \frac{1}{k \vare^k(x) }
+ \frac{c}{ \we^{r+1}(x)} \right)  \frac{ \W(x)}{x^{r+1}}  \le c \frac{ \W(x)}{\vare^k(x) x^{r+1}} ,
\end{align*}
and so
\begin{align*}
\w_k\left(f,\vare(x)x^{(r+1)/k}; [0,1]\right) & \le c \W(x) +c \we^{r+1}(x)x^{r+1}\|f\| \\
& \le c(1+ \W^{-1}(1)\norm{f}{[0,1]}) \W(x),
\end{align*}
which, together with \ineq{87}, yields \ineq{01}.

\lem{negative1} is now proved in all cases.

\begin{remark}
Note that the proof in Case~(iii) works for $k=1$ and $r\ge 1$ as well. Hence, except for $(k,r)=(1,0)$, Case~(i) is covered by Case~(iii).
\end{remark}

\subsection{Weak negative theorem} \label{secweak}

The following negative theorem  has a much simpler proof than \thm{negative}, but it is not as powerful.

\begin{theorem}[weak negative theorem]\label{weaknegative}
Let $k\in\N$,  $r\in\N_0$, $x_0\in I$, and let a positive function $\vare \in C(0,1]$ be such that $\lim_{x\to 0^+} \vare(x)=0$.
Then, for any $n\in\N$, $M>0$ and $\delta>0$, there exists $F=F_{n, M,\delta} \in C^r$ such that, for any $P_n\in\Pn$, we have
\be \label{weneg}
\sup_{x\in [x_0-\delta, x_0+\delta]\cap I} \frac{|F(x)-P_n(x)|}{\omega_k\left(F^{(r)},\vare(|x-x_0|) |x-x_0|^{(r+1)/k}\right)} \ge M.
\ee
\end{theorem}

The main shortcoming of this theorem is that it does not exclude a possibility that the left-hand side of \ineq{weneg} is uniformly bounded above for all functions $F\in C^r$ if the degree of approximating polynomials is allowed to be sufficiently large depending on $F$ (\ie if $n\ge N(f)$, for some natural number $N$ that is allowed to depend on $f$).
Of course, \thm{negative} shows that this is impossible.

\begin{proof}
The idea of the proof  is essentially the same  as in \cites{Yu} (see also \cite{K-sim}*{Theorem 3}).

As in the proof of \thm{negative}, without loss of generality, one can assume that $x_0=0$ and $I$ is replaced by $[0,1]$.
Thus, we will show
that, for any $n\in\N$, $M>0$ and $\delta>0$, there exists sufficiently small $\eps >0$ such that, for $F(x) := (\eps - x)_+^{k+r}$ and any
any $P_n\in\Pn$, we have
\be \label{weakin}
\sup_{x\in [0, \delta]} \frac{|F(x)-P_n(x)|}{\omega_k\left(F^{(r)},\vare(x) x^{(r+1)/k}; [0,1]\right)} \ge M.
\ee
Clearly,  $F \in C^{k+r-1}[0,1]$,   $\norm{F^{(k+r)}}{L_\infty[0,1]} \le c(k,r)$,  and $\norm{F^{(\nu)}}{[0,1]} \sim \eps^{k+r-\nu}$, $0\le \nu \le k+r-1$.
Hence,  $\w_k(F^{(r)},t; [0,1]) \le c \min\{t^k, \eps^k\}$.

Suppose that, for some $P_n\in\Pn$, \ineq{weakin} is not true. Then,  we must have
\[
 |F(x)-P_n(x)| \le M \omega_k\left(F^{(r)},\vare(x) x^{(r+1)/k}; [0,1]\right) , \quad x\in [0,\delta].
\]
Hence,
\[
 |F(x)-P_n(x)| \le c \vare^k(x) x^{r+1} ,
\]
and so $P_n^{(r+1)}(0) = F^{(r+1)}(0) \sim \eps^{k-1}$. Also,
 \[
 |P_n(x)| \le |F(x)|+|F(x)-P_n(x)| \le c  \eps^{k}, \quad x\in [0,\delta].
 \]
Hence, by Markov's inequality $\norm{P_n^{(r+1)}}{[0,\delta]} \le c(k,r,M,n,\delta) \eps^{k}$, and we get a contradiction by picking $\eps$ to be sufficiently small.
\end{proof}


\sect{Applications} \label{sec6}

\subsection{One estimate for $q$-monotone functions} 

It is clear that, given a function $f$ and a polynomial $P_n$ approximating it on $I$, one can achieve interpolation of $f$ at $\pm 1$ by adding a linear polynomial to $P_n$. This is a particular (and trivial) instance of a more general standard approach involving so-called Boolean sums that is often used in the literature (see,  \eg  \cites{D,CG89,De76,CG}).
The proof of the following lemma is trivial, and we only state it here for the sake of reader  convenience.
At the same time,  this lemma immediately implies \cor{corqmon} which is new and  quite far from being obvious, and  would be rather difficult to prove directly without employing \thm{main}.

\begin{lemma} \label{corzero}
Let $k, n\in\N$   and  $f\in C$, and suppose that $P_n\in\Pn$ is such that
$|f(x)-P_n(x)|\le A \omega_k(f ,\rho_n(x))$,  $x\in I$,
where $A\ge 1$ is some  constant.
Then, there exists a polynomial $Q_n\in \Pn$ such that, $Q_n'' \equiv P_n''$ and,  for each $1\le \ell \le k$,
\[
|f(x)-Q_n(x)|\le c(k) A \omega_\ell(f ,\varphi^{2/\ell}(x)n^{-2+ 2/\ell}),\quad \text{if} \quad   1-n^{-2}\le| x|\le1.
\]
\end{lemma}

\begin{proof}
Let $L(g,\cdot)$ be the linear polynomial interpolating $g$ at $\pm 1$, \ie
$L(g,x) =   (1+x)g(1)/2 +  (1-x)g(-1)/2$, and define
$Q_n(x) := P_n(x) + L(f-P_n, x)$.
Then, $Q_n''\equiv P_n''$,  $Q_n(\pm 1) = f(\pm 1)$ and
\begin{align*}
|f(x)-Q_n(x)| & \le   |f(x)-P_n(x)| + |f(1)-P_n(1)| + |f(-1)-P_n(-1)| \\
& \le   A \omega_k(f ,\rho_n(x)) + 2 A \omega_k(f ,n^{-2}) \le 3A  \omega_k(f ,\rho_n(x)) .    
\end{align*}
It remains to apply \thm{main} with $r=0$ in order to finish the proof.
\end{proof}

We say that a function $f\in C$ is $q$-monotone on $I$ if $\Delta_u^q(f,x) \ge 0$ for all $u>0$ and $x\in I$,
 and denote the set of all $q$-monotone (continuous) functions by $\Delta^{(q)}$.
 In particular, $\Delta^{(1)}$ and $\Delta^{(2)}$
are, respectively, the classes of all nondecreasing and convex
functions from $C$.

 \lem{corzero}   implies that non-interpolatory pointwise estimates for $q$-monotone polynomial approximation (see,  \eg \cite{KLPS} for discussions)   imply
 interpolatory ones if $q\ge 2$.
In particular, the following result holds. (In the case $q=1$, we use \cite{DY} and \thm{main}.)

\begin{corollary} \label{corqmon}
Let $q\in\N$   and  $f\in C  \cap\Delta^{(q)}$. Then, for every $n\in\N$, there exists a polynomial $P_n \in\Pn \cap \Delta^{(q)}$ such that
\[
|f(x)-P_n(x)| \le c(q)    \w_1\left(f, \min\{ \varphi^2(x), n^{-1}\varphi(x)\} \right)   .
\]
\end{corollary}

\subsection{Polynomial  approximation with Hermite interpolation  of Sobolev and Lipschitz classes
}

Throughout this section, we assume that $Y$ consists of distinct points in $I$ with the multiplicity $r+1$ each, and recall that
$Z= Z(Y) = \{z_j\}_{j=0}^{\mu-1}$ is the subset  of all  distinct points in $Y$, and $\delta(Z)$ denotes the smallest distance among the points in $Z$, \ie  $\delta(Z) := \min_{0\le i  \le \mu-2} (z_{i+1}-z_i)$.

Recall that $W^r$ denotes the space of   functions on $[-1,1]$ for which $f^{(r-1)}$ is  absolutely continuous   and
$\norm{f^{(r)}}{\infty}    < \infty$, where
 $\norm{\cdot}{\infty}$ is the essential supremum  on $I$.

We have the following theorem on best interpolatory estimates of functions from $W^r$ by polynomials that immediately follows from \cor{19}.

\begin{theorem} 
Let $r, \mu \in\N$,  and  let $Z=\{z_j\}_{j=0}^{\mu-1} \subset I$ be a set of $\mu$ distinct points. 
Then, for every $f\in W^r$  and $n\ge N(r,\mu, \delta(Z))$,
\be\label{estwr1}
 \inf_{P_n\in \Pn } \norm{ \frac{f-P_n}{ (\min\{\rho_n(x), \dist(x, Z) \})^r} }{\infty} \le c(r,\mu) \norm{f^{(r)}}{\infty}  .
\ee
\end{theorem}

\begin{remark}
It follows from \lem{umzhlem} that \ineq{estwr1} is also true for all  $n\ge \mu r-1$ if the constant $c$ is allowed to depend on $\delta(Z)$.
\end{remark}

The following lemma
shows that the quantity $\min\{\rho_n(x), \dist(x, Z) \}$ in \ineq{estwr1} is exact in the sense that one cannot improve the rate of approximation near any of the points in $Z$.

\begin{lemma} 
For any $r\in\N$ and $z\in I$  there exist  a function $f\in W^r$ and a positive constant $c_0=c_0(r)$ such that, for any $n\in\N$ and
 any   $P_n\in\Pn$,
\be \label{in32}
\limsup_{x\to z} \frac{|f(x)-P_n(x)|}{|x-z|^{r}  } \ge  c_0  \norm{f^{(r)}}{\infty} .
\ee
\end{lemma}

\begin{proof} Without loss of generality, we can assume that $z\le 0$.
If    $f(x) := (x-z)_+^{r} \cos(2\pi \ln|x-z|)$, 
then  $f\in W^r$ and $\norm{f^{(r)}}{\infty} \sim 1$.
We will show that \ineq{in32} holds with $c_0 := \norm{f^{(r)}}{\infty}^{-1}$.
Let $P_n$ be an arbitrary polynomial from $\Pn$, and
since $P_n$ has at most $n$ zeros, there
exists $\eps>0$ such that   $P_n(x)$ is either nonnegative or nonpositive on $[z,z+\eps]$. Without loss of generality suppose that $P_n(x) \le 0$, $z\le x \le z+\eps$.
Then
\[
|f(x_m)-P_n(x_m)| \ge  f(x_m) = (x_m-z)^{r} , \quad \mbox{\rm for all }\;  x_m := z+e^{-m},
\]
  where $m\in\N$, $m\ge |\ln \eps|$, and \ineq{in32} follows.
 \end{proof}

We now recall that
  $\Lip^*\alpha $ denotes the space of all functions $f$ on $I$ such that the seminorm $\sn := \sup_{t>0} \left( t^{\nu-\alpha} \w_2(f^{(\nu)}, t) \right) < \infty$,
where $\nu :=   \lceil\alpha\rceil - 1$.
Together with the classical inverse theorems (see,  \eg  \cite{DL}*{Theorem 8.6.1}), \ineq{classdir} implies that,  if $\alpha>0$, then a function $f$ is in $\Lip^*\alpha$ if and only if
 \[ 
 \inf_{P_n\in\Pn} \norm{\rho_n^{-\alpha}  (f -P_n )}{} =  O(1) .
 \]

 Given $\alpha >0$,  for $f \in \Lip^*\alpha$, we define $g := f-L$, where $L$ is a  polynomial from $\Poly_{\nu+1}$ such that $L^{(\nu)}$ interpolates $f^{(\nu)}$ at the endpoints of $I$. Then,
by the Marchaud inequality, if $\alpha \not\in\N$, we have
\begin{align*}
\w_1(g^{(\nu)}, t) & \le c(\alpha) \left(  t^{\alpha-\nu} \sn +   t \norm{g^{(\nu)}}{} \right) \\
&  \le  c(\alpha)    t^{\alpha-\nu} \sn + c(\alpha) t \w_2(f^{(\nu)}, 1)    \le c(\alpha)   t^{\alpha-\nu} \sn,
\end{align*}
and, if $\alpha \in \N$ (and so $\nu=\alpha-1$),
\[
\w_1(g^{(\nu)}, t) \le c(\alpha)   t |\ln t| \,  \sn.
\]

The following estimate for functions from $\Lip^*\alpha$ classes immediately follows from \cor{19}.

\begin{corollary} 
Let $Z=\{z_j\}_{j=0}^{\mu-1} \subset I$ be a set of $\mu$ distinct points.
If $\alpha >0$ and
 $f \in \Lip^*\alpha$, then for any $n\ge N(\alpha, \mu, \delta(Z))$,
there exists $P_n\in\Pn$ such that, for all $x\in I$,
\be \label{corin}
|f(x)-P_n(x)| \le
c(\alpha, \mu) \sn  \, \cdot \,
\begin{cases}
 \left(  \min\{ \dist(x, Z),  \rho_n(x) \}  \right)^\alpha   ,  &  \text{if }\; \alpha \not\in\N, \\
\left(  \min \{  \DDD (x)   , \rho_n (x)  \}  \right)^\alpha  , & \text{if }\; \alpha \in \N,
\end{cases}
\ee
where $\DDD (x) :=  \dist (x, Z)   \left|\ln \big( \dist(x, Z)/3 \big) \right|^{1/\alpha}$.
\end{corollary}

\begin{remark}
It follows from \lem{umzhlem} that \ineq{corin} is also true for all  $n\ge \mu \lceil \alpha\rceil-1$ if the constant $c$ is allowed to depend on $\delta(Z)$.
\end{remark}

The following lemma shows   that  the estimate \ineq{corin} is exact in the sense that one cannot expect a better rate of approximation near points in $Z$.

\begin{lemma} 
For any $\alpha>0$ and $z\in I$,  there exist a function $f\in\Lip^*\alpha$ and a positive constant $c_0=c_0(\alpha)$ such that,
 for any $n\in\N$ and
 any   $P_n\in\Pn$,
\[
\limsup_{x\to z} \frac{|f(x)-P_n(x)|}{|x-z|^{\alpha}   |\psi_\alpha(x)|} \ge c_0 \sn ,
\]
where $\psi_\alpha(x) :=  \ln|x-z|$ if $\alpha\in\N$, and $\psi_\alpha(x) := 1$ if $\alpha\not\in\N$.
\end{lemma}

\begin{proof} Without loss of generality, we can assume that $z\le 0$.
If  $f(x) := (x-z)_+^{\alpha}\psi_\alpha(x)$, 
then  $\sn \le c (\alpha)$, and set $c_0 := \sn^{-1}/2$.
If the claim of the lemma is not true, then
there exists a polynomial
$P_n\in \Pn$ and $\eps>0$    such that
\be \label{newineq}
 |f(x)-P_n(x)| \le \frac12 (x-z)^{\alpha}  |\psi_\alpha(x)|  , \quad x\in [z,z+\eps].
\ee
Then,   \ineq{newineq} implies that  $f^{(i)}(z) = P_n^{(i)}(z)=0$, for all $0\le i \le \nu:= \lceil \alpha \rceil - 1$. Now, by Markov's inequality, we have
$\norm{P_n^{(\nu+1)}}{}   \le c(n,\nu)  \norm{P_n}{} =: A$, and so
\[
|P_n(x)| = \frac{1}{\nu!} \left| \int_{z}^x (x-t)^\nu P_n^{(\nu+1)}(t) \, dt \right| \le \frac{A}{(\nu+1)!} (x-z)^{\nu+1} , \quad x\in I .
\]
Hence, for $x>z$,
\[
\frac{|f(x)-P_n(x)|}{(x-z)^{\alpha} |\psi_{\alpha}(x)|} \ge 1 - \frac{|P_n(x)|}{(x-z)^{\alpha} |\psi_{\alpha}(x)|}
 \ge
  1 - \frac{A}{(\nu+1)!} \cdot  \frac{(x-z)^{\lceil \alpha \rceil -\alpha}}{|\psi_{\alpha} (x)|}  \to 1
\]
as $x\to z^+$,  which contradicts \ineq{newineq}.
\end{proof}


\begin{bibsection}
\begin{biblist}

\bib{AB}{article}{
   author={Andrievskii, V. V.},
   author={Blatt, H.-P.},
   title={Polynomial approximation of functions on a quasi-smooth arc with
   Hermitian interpolation},
   journal={Constr. Approx.},
   volume={30},
   date={2009},
   number={1},
   pages={121--135},
}

\bib{BK95}{article}{
   author={Bal\'{a}zs, K.},
   author={Kilgore, T.},
   title={On some constants in simultaneous approximation},
   journal={Internat. J. Math. Math. Sci.},
   volume={18},
   date={1995},
   number={2},
   pages={279--286},
}

\bib{B59}{article}{
   author={Brudny\u{\i}, Yu. A.},
   title={Approximation by integral functions on the exterior of a segment
   or on a semi-axis},
   language={Russian},
   journal={Dokl. Akad. Nauk SSSR},
   volume={124},
   date={1959},
   pages={739--742},
}

\bib{B}{article}{
   author={Brudny\u{\i}, Yu. A.},
   title={Generalization of a theorem of A. F. Timan},
   language={Russian},
   journal={Dokl. Akad. Nauk SSSR},
   volume={148},
   date={1963},
   pages={1237--1240},
}

\bib{BG}{article}{
   author={Brudnyi, Yu. A.},
   author={Gopengauz, I. E.},
   title={On an approximation family of discrete polynomial operators},
   journal={J. Approx. Theory},
   volume={164},
   date={2012},
   number={7},
   pages={938--953},
}


\bib{CG89}{article}{
   author={Cao, J. D.},
   author={Gonska, H. H.},
   title={Computation of DeVore-Gopengauz-type approximants},
   conference={
      title={Approximation theory VI, Vol. I},
      address={College Station, TX},
      date={1989},
   },
   book={
      publisher={Academic Press, Boston, MA},
   },
   date={1989},
   pages={117--120},
}


\bib{CG}{article}{
   author={Cao, J. D.},
   author={Gonska, H. H.},
   title={Approximation by Boolean sums of positive linear operators. II.
   Gopengauz-type estimates},
   journal={J. Approx. Theory},
   volume={57},
   date={1989},
   number={1},
   pages={77--89},
}

\bib{Da}{article}{
   author={Dahlhaus, R.},
   title={Pointwise approximation by algebraic polynomials},
   journal={J. Approx. Theory},
   volume={57},
   date={1989},
   number={3},
   pages={274--277},
}


\bib{D}{article}{
   author={DeVore, R. A.},
   title={Pointwise approximation by polynomials and splines},
   conference={
      title={The theory of the approximation of functions},
      address={Proc. Internat. Conf., Kaluga},
      date={1975},
   },
   book={
      publisher={``Nauka'', Moscow},
   },
   date={1977},
    pages={132--141},
}


\bib{De76}{article}{
   author={DeVore, R. A.},
   title={Degree of approximation},
   conference={
      title={Approximation theory, II},
      address={Proc. Internat. Sympos., Univ. Texas, Austin, Tex.},
      date={1976},
   },
   book={
      publisher={Academic Press, New York},
   },
   date={1976},
    pages={117--161},
}




\bib{DL}{book}{
   author={DeVore, R. A.},
   author={Lorentz, G. G.},
   title={Constructive approximation},
   series={Grundlehren der Mathematischen Wissenschaften [Fundamental
   Principles of Mathematical Sciences]},
   volume={303},
   publisher={Springer-Verlag, Berlin},
   date={1993},
   pages={x+449},
}



\bib{DY}{article}{
   author={DeVore, R. A.},
   author={Yu, X. M.},
   title={Pointwise estimates for monotone polynomial approximation},
   journal={Constr. Approx.},
   volume={1},
   date={1985},
   number={4},
   pages={323--331},
}




\bib{DJ}{article}{
   author={Ditzian, Z.},
   author={Jiang, D.},
   title={Approximation of functions by polynomials in $C[-1,1]$},
   journal={Canad. J. Math.},
   volume={44},
   date={1992},
   number={5},
   pages={924--940},
}






\bib{D56}{article}{
   author={Dzyadyk, V. K.},
   title={Constructive characterization of functions satisfying the
   condition ${\rm Lip}\,\alpha(0<\alpha<1)$ on a finite segment of
   the real axis},
   language={Russian},
   journal={Izv. Akad. Nauk SSSR. Ser. Mat.},
   volume={20},
   date={1956},
   pages={623--642},
}

\bib{D-1958}{article}{
   author={Dzyadyk, V. K.},
   title={A further strengthening of Jackson's theorem on the approximation
   of continuous functions by ordinary polynomials},
   language={Russian},
   journal={Dokl. Akad. Nauk SSSR},
   volume={121},
   date={1958},
   pages={403--406},
}




\bib{D59}{article}{
   author={Dzyadyk, V. K.},
   title={On a problem of S. M. Nikol'ski\u{\i} in a complex region},
   language={Russian},
   journal={Izv. Akad. Nauk SSSR. Ser. Mat.},
   volume={23},
   date={1959},
   pages={697--736},
}



\bib{DS}{book}{
   author={Dzyadyk, V. K.},
   author={Shevchuk, I. A.},
   title={Theory of uniform approximation of functions by polynomials},
   publisher={Walter de Gruyter GmbH \& Co. KG, Berlin},
   date={2008},
   pages={xvi+480},
}




\bib{F-1959}{article}{
   author={Freud, G.},
   title={\"{U}ber die Approximation reeller stetigen Funktionen durch
   gew\"{o}hnliche Polynome},
   language={German},
   journal={Math. Ann.},
   volume={137},
   date={1959},
   pages={17--25},
}



\bib{GS}{article}{
   author={Gilewicz, J.},
   author={Shevchuk, I. A.},
   title={Comonotone approximation},
   language={Russian, with English and Russian summaries},
   journal={Fundam. Prikl. Mat.},
   volume={2},
   date={1996},
   number={2},
   pages={319--363},
}

\bib{GH}{article}{
   author={Gonska, H.},
   author={Hinnemann, E.},
   title={Punktweise Absch\"{a}tzungen zur Approximation durch algebraische Polynome},
   language={German},
   journal={Acta Math. Hungar.},
   volume={46},
   date={1985},
   number={3-4},
   pages={243--254},
}



\bib{GLSW}{article}{
   author={Gonska, H. H.},
   author={Leviatan, D.},
   author={Shevchuk, I. A.},
   author={Wenz, H.-J.},
   title={Interpolatory pointwise estimates for polynomial approximation},
   journal={Constr. Approx.},
   volume={16},
   date={2000},
   number={4},
   pages={603--629},
}



\bib{Gop}{article}{
   author={Gopengauz, I. E.},
   title={On a theorem of A. F. Timan on the approximation of functions by
   polynomials on a finite interval},
   language={Russian},
   journal={Mat. Zametki},
   volume={1},
   date={1967},
   pages={163--172},
}

\bib{HG}{article}{
   author={Hinnemann, E.},
   author={Gonska, H. H.},
   title={Generalization of a theorem of DeVore},
   conference={
      title={Approximation theory, IV},
      address={College Station, Tex.},
      date={1983},
   },
   book={
      publisher={Academic Press, New York},
   },
   date={1983},
   pages={527--532},
}



\bib{hky}{article}{
   author={Hu, Y. K.},
   author={Kopotun, K. A.},
   author={Yu, X. M.},
   title={Constrained approximation in Sobolev spaces},
   journal={Canad. J. Math.},
   volume={49},
   date={1997},
   number={1},
   pages={74--99},
}



\bib{KP96}{article}{
   author={Kilgore, T.},
   author={Prestin, J.},
   title={Pointwise Gopengauz estimates for interpolation},
   journal={Ann. Univ. Sci. Budapest. Sect. Comput.},
   volume={16},
   date={1996},
   pages={253--261},
}

\bib{KS86}{article}{
   author={Kis, O.},
   author={Szabados, J.},
   title={On some de la Vall\'{e}e-Poussin type discrete linear operators},
   journal={Acta Math. Hungar.},
   volume={47},
   date={1986},
   number={1-2},
   pages={239--260},
}






\bib{K-sim}{article}{
   author={Kopotun, K. A.},
   title={Simultaneous approximation by algebraic polynomials},
   journal={Constr. Approx.},
   volume={12},
   date={1996},
   number={1},
   pages={67--94},
}




\bib{KLPS}{article}{
   author={Kopotun, K. A.},
   author={Leviatan, D.},
   author={Prymak, A.},
   author={Shevchuk, I. A.},
   title={Uniform and pointwise shape preserving approximation by algebraic  polynomials},
   journal={Surv. Approx. Theory},
   volume={6},
   date={2011},
   pages={24--74},
}



\bib{KLS-cjm}{article}{
   author={Kopotun, K. A.},
   author={Leviatan, D.},
   author={Shevchuk, I. A.},
   title={Convex polynomial approximation in the uniform norm: conclusion},
   journal={Canad. J. Math.},
   volume={57},
   date={2005},
   number={6},
   pages={1224--1248},
}


\bib{KLSUMZh}{article}{
   author={Kopotun, K. A.},
   author={Leviatan, D.},
   author={Shevchuk, I. A.},
   title={On one estimate of divided differences and its applications},
   language={English, with English and Ukrainian summaries},
   journal={Ukra\"{\i}n. Mat. Zh.},
   volume={71},
   date={2019},
   number={2},
   pages={230--245},
   translation={
      journal={Ukrainian Math. J.},
      volume={71},
      date={2019},
      number={2},
      pages={259--277},
   },
}



\bib{KLSconspline}{article}{
   author={Kopotun, K. A.},
   author={Leviatan, D.},
   author={Shevchuk, I. A.},
   title={Interpolatory estimates for convex piecewise polynomial
   approximation},
   journal={J. Math. Anal. Appl.},
   volume={474},
   date={2019},
   number={1},
   pages={467--479},
}




\bib{L57}{article}{
   author={Lebed', G. K.},
   title={Inequalities for polynomials and their derivatives},
   language={Russian},
   journal={Dokl. Akad. Nauk SSSR (N.S.)},
   volume={117},
   date={1957},
   pages={570--572},
}




\bib{Li}{article}{
   author={Li, W.},
   title={On Timan type theorems in algebraic polynomial approximation},
   language={Chinese},
   journal={Acta Math. Sinica},
   volume={29},
   date={1986},
   number={4},
   pages={544--549},
}




 \bib{Lor}{article}{
   author={Lorentz, G. G.},
   title={Problem},
   conference={
      title={On Approximation Theory (Proceedings of Conference in
      Oberwolfach, 1963)},
   },
   book={
      publisher={Birkh\"{a}user, Basel},
   },
   date={1964},
   pages={185},
}





\bib{Sz}{article}{
   author={Szabados, J.},
   title={Discrete linear interpolatory operators},
   journal={Surv. Approx. Theory},
   volume={2},
   date={2006},
   pages={53--60},
}

\bib{Tel}{article}{
   author={Teljakovski\u{\i}, S. A.},
   title={Two theorems on approximation of functions by algebraic
   polynomials},
   language={Russian},
   journal={Mat. Sb. (N.S.)},
   volume={70 (112)},
   date={1966},
   pages={252--265},
}

\bib{T-1951}{article}{
   author={Timan, A. F.},
   title={A strengthening of Jackson's theorem on the best approximation of
   continuous functions by polynomials on a finite segment of the real axis},
   language={Russian},
   journal={Doklady Akad. Nauk SSSR (N.S.)},
   volume={78},
   date={1951},
   pages={17--20},
}


\bib{Tr62}{article}{
   author={Trigub, R. M.},
   title={Approximation of functions by polynomials with integer
   coefficients},
   language={Russian},
   journal={Izv. Akad. Nauk SSSR Ser. Mat.},
   volume={26},
   date={1962},
   pages={261--280},
}

\bib{Trigub}{article}{
   author={Trigub, R. M.},
   title={A general direct theorem on the approximation of functions in the
   class $C^r$ by algebraic polynomials with Hermite interpolation},
   language={Russian},
   journal={Dokl. Akad. Nauk},
   volume={386},
   date={2002},
   number={5},
   pages={599--601},
}

\bib{TIzv}{article}{
   author={Trigub, R. M.},
   title={Approximation of functions by polynomials with Hermite
   interpolation and with constraints on the coefficients},
   language={Russian, with Russian summary},
   journal={Izv. Ross. Akad. Nauk Ser. Mat.},
   volume={67},
   date={2003},
   number={1},
   pages={199--221},
   issn={1607-0046},
   translation={
      journal={Izv. Math.},
      volume={67},
      date={2003},
      number={1},
      pages={183--206},
      issn={1064-5632},
   },
}



\bib{V}{article}{
   author={V\'{e}rtesi, P.},
   title={Convergent interpolatory processes for arbitrary systems of nodes},
   journal={Acta Math. Acad. Sci. Hungar.},
   volume={33},
   date={1979},
   number={1-2},
   pages={223--234},
}

\bib{XZ}{article}{
   author={Xie, T.},
   author={Zhou, X.},
   title={A modification of Lagrange interpolation},
   journal={Acta Math. Hungar.},
   volume={92},
   date={2001},
   number={4},
   pages={285--297},
}

\bib{Yu}{article}{
   author={Yu, X. M.},
   title={Pointwise estimate for algebraic polynomial approximation},
   journal={Approx. Theory Appl.},
   volume={1},
   date={1985},
   number={3},
   pages={109--114},
}


\end{biblist}
\end{bibsection}
\end{document}